\documentclass[a4paper,oneside,10.5pt]{article}%
\usepackage{amsmath}
\usepackage{amsfonts}
\usepackage{amssymb}
\usepackage{graphicx}%
\providecommand{\U}[1]{\protect\rule{.1in}{.1in}}

\newtheorem{theorem}{Theorem}[section]

\newtheorem{definition}[theorem]{Definition}
\newtheorem{assumption}[theorem]{Assumption}

\newtheorem{lemma}[theorem]{Lemma}

\newtheorem{remark}[theorem]{Remark}

\newenvironment{proof}[1][Proof]{\noindent \textbf{#1.} }{\  \rule{0.5em}{0.5em}}
\numberwithin{equation}{section}

\begin{document}

\title{Maximum principle for stochastic optimal control problem of forward-backward
stochastic difference systems}
\author{Shaolin Ji\thanks{Zhongtai Securities Institute for Financial Studies,
Shandong University, Jinan, Shandong 250100, PR China. jsl@sdu.edu.cn.
Research supported by NSF (No. 11571203).}
\and Haodong Liu\thanks{Zhongtai Securities Institute for Financial Studies,
Shandong University, Jinan, Shandong 250100, PR China. (Corresponding
author).}}
\maketitle

\textbf{Abstract}: In this paper, we study the maximum principle for
stochastic optimal control problems of forward-backward stochastic difference
systems (FBS$\Delta$Ss). Two types of FBS$\Delta$Ss are investigated. The
first one is described by a partially coupled {forward-backward stochastic
difference equation (}FBS$\Delta$E) and the second one is described by a fully
coupled FBS$\Delta$E. By adopting an appropriate representation of the product
rule and an appropriate formulation of the {backward stochastic difference
equation (}BS$\Delta$E), we deduce the adjoint difference equation. Finally,
the maximum principle for this optimal control problem with the control domain
being convex is established.

{\textbf{Keywords}: backward stochastic difference equations; forward-backward
stochastic difference equations; monotone condition; stochastic optimal
control; }maximum principle

\addcontentsline{toc}{section}{\hspace*{1.8em}Abstract}

\section{Introduction}

The Maximum Principle is one of the principal approaches in solving the
optimal control problems. A lot of work has been done on the Maximum Principle
for forward stochastic system. See, for example, Bensoussan \cite{b82}, Bismut
\cite{b78}, Kushner \cite{k72}, Peng \cite{p90}. Peng also firstly studied one
kind of forward-backward stochastic control system (FBSCS) in \cite{p93} and
obtained the maximum principle for this kind of control system with control
domain being convex. The FBSCSs have wide applications in many fields. As the
stochastic differential recursive utility, which is a generalization of a
standard additive utility, can be regarded as a solution of a {backward
stochastic differential equation (}BSDE). The recursive utility optimization
problem can be described by a optimization problem for a FBSCS (see
\cite{ss99}). Besides, in the dynamic principal-agent problem with
unobservable states and actions, the principal's problem can be formulated as
a partial information optimal control problem of a FBSCS (see \cite{w09}). We
refer to \cite{dz99}, \cite{hwx09}, \cite{lz01},\ \cite{ww09}, \cite{w13},
\cite{x95}\ for other works on optimization problems for FBSCSs.

In this paper, we will discuss the Maximum Principle for optimal control of
discrete time systems described by forward-backward stochastic difference
equations (FBS$\Delta$Es). To the best of our knowledge, there are few results
on such optimization control problems. In fact, the discrete time control
systems are of great value in practice. For example, the digital control can
be formulated as discrete time control problems, where the sampled data is
obtained at discrete instants of time. Besides, the forward-backward
stochastic difference system (FBS$\Delta$S) can be used for modeling in
financial markets. For example, the solution to the backward stochastic
difference equation (BS$\Delta$E) can be used to construct time-consistent
nonlinear expectations (see \cite{ce10+}, \cite{ce11}) and be used for pricing
in the financial markets (see \cite{bcc15}). However, the formulation of
BS$\Delta$E is quite different from its continuous time counterpart. Many
works are devoted to the study of BS$\Delta$Es (see, e.g. \cite{bcc15},
\cite{ce10+}, \cite{ce11}, \cite{s10}). Based on the driving process, there
are mainly two types of formulations of BS$\Delta$Es. One is driving by a
finite state process which takes values from the basis vectors (as in
\cite{ce10+}) and the other is driving by a martingale with independent
increments (as in \cite{bcc15}). For the latter case, the solution of the
BS$\Delta$E is a triple of processes which is due to the discrete time version
of the Kunita--Watanabe decomposition. In this paper, we adopt the second type
of formulation to investigate the optimization problems for FBS$\Delta$Ss.

Let\ $\left(  \Omega,\mathcal{F},\left\{  \mathcal{F}_{t}\right\}  _{0\leq
t\leq T},P\right)  $\ be a probability space, and\ $W_{t}$ be a martingale
process with independent increments. Define the difference operator $\Delta$
as $\Delta V_{t}=V_{t+1}-V_{t}$. Here we consider two types of controlled
FBS$\Delta$Ss.

Problem 1 (partially coupled system):

The controlled system is%
\begin{equation}
\left\{
\begin{array}
[c]{rcl}%
\Delta X_{t} & = & b\left(  t,X_{t},u_{t}\right)  +\sum_{i=1}^{d}\sigma
_{i}\left(  t,X_{t},u_{t}\right)  \Delta W_{t}^{i},\\
X_{0} & = & x_{0},\\
\Delta Y_{t} & = & -f\left(  t+1,X_{t+1},Y_{t+1},Z_{t+1},u_{t+1}\right)
+Z_{t}\Delta W_{t}+\Delta N_{t},\\
Y_{T} & = & y_{T},
\end{array}
\right.  \label{c5_eq_w_pc_state_eq}%
\end{equation}

and the cost functional is%
\begin{equation}
J\left(  u\left(  \cdot\right)  \right)  =\mathbb{E}\left[  \sum_{t=0}%
^{T-1}l\left(  t,X_{t},Y_{t},Z_{t},u_{t}\right)  +h\left(  X_{T}\right)
\right]  . \label{c5_eq_w_pc_cost_eq}%
\end{equation}

Problem 2 (fully coupled system):

The controlled system is:%
\begin{equation}
\left\{
\begin{array}
[c]{rcl}%
\Delta X_{t} & = & b\left(  t,X_{t},Y_{t},Z_{t},u_{t}\right)  +\sum_{i=1}%
^{d}\sigma_{i}\left(  t,X_{t},Y_{t},Z_{t},u_{t}\right)  \Delta W_{t}^{i},\\
X_{0} & = & x_{0},\\
\Delta Y_{t} & = & -f\left(  t+1,X_{t+1},Y_{t+1},Z_{t+1},u_{t+1}\right)
+Z_{t}\Delta W_{t}+\Delta N_{t},\\
Y_{T} & = & y_{T},
\end{array}
\right.  \label{fbsde_state_eq}%
\end{equation}

and the cost functional is%
\begin{equation}
J\left(  u\left(  \cdot\right)  \right)  =\mathbb{E}\left[  \sum_{t=0}%
^{T-1}l\left(  t,X_{t},Y_{t},Z_{t},u_{t}\right)  +h\left(  X_{T}\right)
\right]  . \label{fbsde_cost_eq}%
\end{equation}

Let $\left\{  U_{t}\right\}  _{t\in\left\{  0,1,...,T-1\right\}  }$ be a
sequence of nonempty convex subset of $\mathbb{R}^{r}$. We denote the set of
admissible controls $\mathcal{U}$ by $\mathcal{U}=\left\{  u\left(
\cdot\right)  \in\mathcal{M}^{2}\left(  0,T-1;\mathbb{R}^{r}\right)  |u\left(
t\right)  \in U_{t}\right\}  .$ It can be seen that in Problem 1, $b$ and
$\sigma$ do not contain the solution $(Y,Z)$ of the backward equation. This
kind of FBS$\Delta$E is called the partially coupled FBS$\Delta$E. Meanwhile,
the system in Problem 2 is called the fully coupled FBS$\Delta$E.

The optimal control problem is to find the optimal control $u\in\mathcal{U}$,
such that the optimal control and the corresponding state trajectory can
minimize the cost functional $J\left(  u\left(  \cdot\right)  \right)  $. In
this paper, we assume the control domain is convex. By making the perturbation
of the optimal control at a fixed time point, we obtain the maximum principle
for problem 1 and 2.

To build the maximum principle, the key step is to find the adjoint variables
which can be applied to deduce the variational inequality. In \cite{lz15}, the
authors studied the maximum principle for a discrete time stochastic optimal
control problem in which the state equation is only governed by a forward
stochastic difference equation. By applying the Riesz representation theorem,
they explicitly obtained the adjoint variables and establish the maximum
principle. But to solve our problems, we need to construct the adjoint
difference equations since generally the adjoint variables can not be obtained
explicitly for our case. To construct the adjoint equations in our discrete
time framework, the techniques which are adopted for the continuous time
framework as in \cite{p90,p93} are not appliable. In this paper, we propose
two techniques to deduce the adjoint difference equations. The first one is
that we choose the following product rule:%
\[
\Delta\left\langle X_{t},Y_{t}\right\rangle =\left\langle X_{t+1},\Delta
Y_{t}\right\rangle +\left\langle \Delta X_{t},Y_{t}\right\rangle
\]
where $X_{t}$ (resp. $Y_{t}$) subjects to a forward (resp. backward)
stochastic difference equation. The second one is that the BS$\Delta$E should
be formulated as in (\ref{bsde_form1}). In other words, the generator $f$ of
the BS$\Delta$E (\ref{bsde_form1}) depends on time $t+1$. It is worth pointing
out that this kind of formulation is just the formulation of the adjoint
equations for stochastic optimal control problems (see \cite{lz15} for the
classical case). Based on these two techniques, we can deduce the adjoint
difference equations. The readers may refer to Remark \ref{re-product-rule}
for more details.

The remainder of this paper is organized as follows. In section 2, two types
of the controlled FBS$\Delta$Ss are formulated. We deduce the maximum
principle for the partially coupled controlled FBS$\Delta$S in section 3.
Finally, we establish the maximum principle for the fully coupled controlled
FBS$\Delta$S in section 4.

\section{Preliminaries and model formulation}

Let $T$ be a deterministic terminal time, and let $\mathcal{T}:=\left\{
0,1,...,T\right\}  $. Consider a filtered probability space $\left(
\Omega,\mathcal{F},\left\{  \mathcal{F}_{t}\right\}  _{0\leq t\leq
T},P\right)  $, with $\mathcal{F}_{0}=\left\{  \emptyset,\Omega\right\}  $ and
$\mathcal{F=F}_{T}$. Here we define the difference operator $\Delta$ as
$\Delta U_{t}=U_{t+1}-U_{t}$. Let $W$ be a fixed $\mathbb{R}^{d}$-valued
square integrable martingale process with independent increments, i.e.
$\mathbb{E}\left[  \Delta W_{t}|\mathcal{F}_{t}\right]  =\mathbb{E}\left[
\Delta W_{t}\right]  =0$ for any $t\in\left\{  0,...,T-1\right\}  $. Also we
suppose that $\mathbb{E}\left[  \Delta W_{t}\left(  \Delta W_{t}\right)
^{\ast}\right]  =I_{d}$ for any $t\in\left\{  0,...,T-1\right\}  $. Here
$\left(  \cdot\right)  ^{\ast}$ denotes vector transposition. We assume that
$\mathcal{F}_{t}$ is the completion of the $\sigma$-algebra generated by the
process $W$ up to time $t$.

Denote by $L^{2}\left(  \mathcal{F}_{t};\mathbb{R}^{n}\right)  $ the set of
all $\mathcal{F}_{t}-$measurable square integrable random variable $X_{t}$
taking values in $\mathbb{R}^{n}$ and by $\mathcal{M}^{2}\left(
0,t;\mathbb{R}^{n}\right)  $ the set of all $\left\{  \mathcal{F}_{s}\right\}
_{0\leq s\leq t}$-adapted square integrable process $X$ taking values in
$\mathbb{R}^{n}$. Moreover, we define $e_{i}=\left(
0,0,...,0,1,0,...,0\right)  ^{\ast}\in\mathbb{R}^{n}$ and mention that an
inequality on a vector quantity is to hold componentwise.

Consider the following backward stochastic difference equation (BS$\Delta$E):%

\begin{equation}
\left\{
\begin{array}
[c]{rcl}%
\Delta Y_{t} & = & -f\left(  t+1,Y_{t+1},Z_{t+1}\right)  +Z_{t}\Delta
W_{t}+\Delta N_{t},\\
Y_{T} & = & \eta,
\end{array}
\right.  \label{bsde_form1}%
\end{equation}
where $\eta\in L^{2}\left(  \mathcal{F}_{T};\mathbb{R}^{n}\right)  $,
$f:\Omega\times\left\{  1,2,...,T\right\}  \times\mathbb{R}^{n}\times
\mathbb{R}^{n\times d}\longmapsto\mathbb{R}^{n}$.

\begin{assumption}
\label{generator_assumption}A1. The function $f\left(  t,y,z\right)  $ is
uniformly Lipschitz continuous and independent of $z$ at $t=T$, i.e. there
exists constants $c_{1},c_{2}>0$, such that for any $t\in\left\{
1,2,...,T-1\right\}  $, $y_{1},y_{2}\in\mathbb{R}^{n}$, $z_{1},z_{2}%
\in\mathbb{R}^{n\times d}$,%
\begin{align*}
\left\vert f\left(  T,y_{1},z_{1}\right)  -f\left(  T,y_{2},z_{2}\right)
\right\vert  &  \leq c_{1}\left\vert y_{1}-y_{2}\right\vert ,\\
\left\vert f\left(  t,y_{1},z_{1}\right)  -f\left(  t,y_{2},z_{2}\right)
\right\vert  &  \leq c_{1}\left\vert y_{1}-y_{2}\right\vert +c_{2}\left\Vert
z_{1}-z_{2}\right\Vert ,\text{ }P-a.s.
\end{align*}

A2. $f\left(  t,0,0\right)  \in L^{2}\left(  \mathcal{F}_{t};\mathbb{R}%
^{n}\right)  $ for any $t\in\left\{  1,2,...,T\right\}  $.
\end{assumption}

\begin{remark}
The BS$\Delta$E (\ref{bsde_form1}) is analogous to the continuous time BSDE
driven by a general martingale (cf. \cite{eh97}), and the solution is a triple
of processes.
\end{remark}

\begin{definition}
A solution to BS$\Delta$E (\ref{bsde_form1}) is a triple of processes $\left(
Y,Z,N\right)  \in\mathcal{M}^{2}\left(  0,T;\mathbb{R}^{n}\right)
\times\mathcal{M}^{2}\left(  0,T-1;\mathbb{R}^{n\times d}\right)
\times\mathcal{M}^{2}\left(  0,T;\mathbb{R}^{n}\right)  $ which satisfies
equality (\ref{bsde_form1}) for all $t\in\left\{  0,1,...,T-1\right\}  $, and
$N$ is a martingale process strongly orthogonal to $W$.
\end{definition}

By using the Galtchouk-Kunita-Watanabe decomposition in \cite{bcc15}, we can
obtain the existence and uniqueness result of BS$\Delta$E (\ref{bsde_form1}):

\begin{theorem}
\label{bsde_result_l}Suppose that Assumption (\ref{generator_assumption})
holds. Then for any terminal condition $\eta\in L^{2}\left(  \mathcal{F}%
_{T};\mathbb{R}^{n}\right)  $, the BS$\Delta$E (\ref{bsde_form1}) has a unique
adapted solution $\left(  Y,Z,N\right)  $.
\end{theorem}

\begin{proof}
We first prove the existence and uniqueness of $\left(  Y_{T-1},Z_{T-1},\Delta
N_{T-1}\right)  $. Due to Assumption (\ref{generator_assumption}) and $\eta\in
L^{2}\left(  \mathcal{F}_{T};\mathbb{R}^{n}\right)  $, we get $f\left(
T,\eta\right)  \in L^{2}\left(  \mathcal{F}_{T};\mathbb{R}^{n}\right)  $. Here
we omit the variable $Z$ since $f$ is independent of $Z$ at time $T$. Then we
have $\mathbb{E}\left[  \left\vert \mathbb{E}\left[  \eta+f\left(
T,\eta\right)  |\mathcal{F}_{T-1}\right]  \right\vert ^{2}\right]  <\infty.$
Hence, $\eta+f\left(  T,\eta\right)  -\mathbb{E}\left[  \eta+f\left(
T,\eta\right)  |\mathcal{F}_{T-1}\right]  $ is a square integrable martingale
difference. So it admits the Galtchouk-Kunita-Watanabe decomposition, which
implies that there exists $Z_{T-1}\in\mathcal{F}_{T-1}$, $Z_{T-1}\Delta
W_{T-1}\in L^{2}\left(  \mathcal{F}_{T};\mathbb{R}^{n}\right)  $, $\Delta
N_{T-1}\in L^{2}\left(  \mathcal{F}_{T};\mathbb{R}^{n}\right)  $ such that
$\mathbb{E}\left[  \Delta N_{T-1}|\mathcal{F}_{T-1}\right]  =0$,
$\mathbb{E}\left[  e_{i}^{\ast}\Delta N_{T-1}\left(  \Delta W_{T-1}\right)
^{\ast}|\mathcal{F}_{T-1}\right]  =0$ and%
\begin{equation}
\eta+f\left(  T,\eta\right)  -\mathbb{E}\left[  \eta+f\left(  T,\eta\right)
|\mathcal{F}_{T-1}\right]  =Z_{T-1}\Delta W_{T-1}+\Delta N_{T-1}.
\label{bsde_proof_1}%
\end{equation}
Moreover, $\Delta N_{T-1}$ is uniquely determined in this decomposition. For
fixed $i\in\left\{  1,2,...,n\right\}  $, premultiply the equation by
$e_{i}^{\ast}$, postmultiply the equation by $\left(  \Delta W_{T-1}\right)
^{\ast}$ and then take the $\mathcal{F}_{T-1}$ conditional expectation. This
yields that
\[
\mathbb{E}\left[  e_{i}^{\ast}\left(  \eta+f\left(  T,\eta\right)  \right)
\left(  \Delta W_{T-1}\right)  ^{\ast}|\mathcal{F}_{T-1}\right]  =e_{i}^{\ast
}Z_{T-1}%
\]
since $\mathbb{E}\left[  \Delta W_{T-1}\left(  \Delta W_{T-1}\right)  ^{\ast
}|\mathcal{F}_{T-1}\right]  =I$. Therefore, we get the unique $Z_{T-1}$ by%
\[
Z_{T-1}=\mathbb{E}\left[  \left(  \eta+f\left(  T,\eta\right)  \right)
\left(  \Delta W_{T-1}\right)  ^{\ast}|\mathcal{F}_{T-1}\right]
\]
and%
\[%
\begin{array}
[c]{ccl}%
\mathbb{E}\left[  \left\Vert Z_{T-1}\right\Vert ^{2}\right]  & \leq &
\mathbb{E}\left[  \mathbb{E}\left[  \left\vert \eta+f\left(  T,\eta\right)
\right\vert ^{2}|\mathcal{F}_{T-1}\right]  \mathbb{E}\left[  \left\vert
\left(  \Delta W_{T-1}\right)  \right\vert ^{2}|\mathcal{F}_{T-1}\right]
\right]  <\infty.
\end{array}
\]
It leads that $Y_{T-1}=\mathbb{E}\left[  \eta+f\left(  T,\eta\right)
|\mathcal{F}_{T-1}\right]  $ and $Y_{T-1}\in L^{2}\left(  \mathcal{F}%
_{T-1};\mathbb{R}^{n}\right)  $.

Then, by similar arguments as above, we can obtain the unique solution
$\left(  Y_{t},Z_{t},\Delta N_{t}\right)  \in L^{2}\left(  \mathcal{F}%
_{t};\mathbb{R}^{n}\right)  \times L^{2}\left(  \mathcal{F}_{t};\mathbb{R}%
^{n\times d}\right)  \times L^{2}\left(  \mathcal{F}_{t};\mathbb{R}%
^{n}\right)  $ for $t\in\left\{  0,1,...,T-2\right\}  .$ Moreover,%
\begin{align*}
Z_{t}  &  =\mathbb{E}\left[  \left(  Y_{t+1}+f\left(  t+1,Y_{t+1}%
,Z_{t+1}\right)  \right)  \left(  \Delta W_{t}\right)  ^{\ast}|\mathcal{F}%
_{t}\right]  ,\\
Y_{t}  &  =\mathbb{E}\left[  Y_{t+1}+f\left(  t+1,Y_{t+1},Z_{t+1}\right)
|\mathcal{F}_{t}\right]  .
\end{align*}
By taking the convention $N_{0}=0$ and letting $N_{t}=N_{0}+\sum_{s=0}%
^{t-1}\Delta N_{s}$, we have that (\ref{bsde_form1}) holds true for all
$t\in\left\{  0,1,...,T-1\right\}  $. Finally, since%
\begin{align*}
&  \mathbb{E}\left[  e_{i}^{\ast}N_{t}\left(  W_{t}\right)  ^{\ast
}|\mathcal{F}_{t-1}\right] \\
&  =e_{i}^{\ast}\sum_{s=0}^{t-2}\Delta N_{s}\mathbb{E}\left[  \left(
W_{t}\right)  ^{\ast}|\mathcal{F}_{t-1}\right]  +\mathbb{E}\left[  e_{i}%
^{\ast}\Delta N_{t-1}\left(  W_{t-1}+\Delta W_{t-1}\right)  ^{\ast
}|\mathcal{F}_{t-1}\right] \\
&  =e_{i}^{\ast}N_{t-1}\left(  W_{t-1}\right)  ^{\ast},
\end{align*}
we conclude that $N$ is strongly orthogonal to $W$.
\end{proof}

Now we consider the control systems (\ref{c5_eq_w_pc_state_eq}%
)-(\ref{c5_eq_w_pc_cost_eq}) and (\ref{fbsde_state_eq})-(\ref{fbsde_cost_eq}).

Let the coefficients in system (\ref{c5_eq_w_pc_state_eq}%
)-(\ref{c5_eq_w_pc_cost_eq}) be such that:%
\begin{align*}
b\left(  \omega,t,x,u\right)   &  :\Omega\times\left\{  0,1,...,T-1\right\}
\times\mathbb{R}^{m}\times\mathbb{R}^{r}\mathbb{\rightarrow R}^{m},\\
\sigma_{i}\left(  \omega,t,x,u\right)   &  :\Omega\times\left\{
0,1,...,T-1\right\}  \times\mathbb{R}^{m}\times\mathbb{R}^{r}%
\mathbb{\rightarrow R}^{m},\\
f\left(  \omega,t,x,y,z,u\right)   &  :\Omega\times\left\{  1,2,...,T\right\}
\times\mathbb{R}^{m}\times\mathbb{R}^{n}\times\mathbb{\mathbb{R}}^{n\times
d}\times\mathbb{R}^{r}\mathbb{\rightarrow R}^{n},\\
l\left(  \omega,t,x,y,z,u\right)   &  :\Omega\times\left\{
0,1,...,T-1\right\}  \times\mathbb{R}^{m}\times\mathbb{R}^{n}\times
\mathbb{\mathbb{R}}^{n\times d}\times\mathbb{R}^{r}\mathbb{\rightarrow R},\\
h\left(  \omega,x\right)   &  :\Omega\times\mathbb{R}^{m}\mathbb{\rightarrow
R}.
\end{align*}
And the coefficients in system (\ref{fbsde_state_eq})-(\ref{fbsde_cost_eq}) be
such that:%
\begin{align*}
b\left(  \omega,t,x,y,z,u\right)   &  :\Omega\times\left\{
0,1,...,T-1\right\}  \times\mathbb{R}^{n}\times\mathbb{R}^{n}\times
\mathbb{\mathbb{R}}^{n\times d}\times\mathbb{R}^{r}\mathbb{\rightarrow R}%
^{n},\\
\sigma_{i}\left(  \omega,t,x,y,z,u\right)   &  :\Omega\times\left\{
0,1,...,T-1\right\}  \times\mathbb{R}^{n}\times\mathbb{R}^{n}\times
\mathbb{\mathbb{R}}^{n\times d}\times\mathbb{R}^{r}\mathbb{\rightarrow R}%
^{n},\\
f\left(  \omega,t,x,y,z,u\right)   &  :\Omega\times\left\{  1,2,...,T\right\}
\times\mathbb{R}^{n}\times\mathbb{R}^{n}\times\mathbb{\mathbb{R}}^{n\times
d}\times\mathbb{R}^{r}\mathbb{\rightarrow R}^{n},\\
l\left(  \omega,t,x,y,z,u\right)   &  :\Omega\times\left\{
0,1,...,T-1\right\}  \times\mathbb{R}^{n}\times\mathbb{R}^{n}\times
\mathbb{\mathbb{R}}^{n\times d}\times\mathbb{R}^{r}\mathbb{\rightarrow R},\\
h\left(  \omega,x\right)   &  :\Omega\times\mathbb{R}^{n}\mathbb{\rightarrow
R}.
\end{align*}

\begin{remark}
The cost functional in \cite{p93} consists of three parts: the running cost
functional, the terminal cost functional of $X_{T}$, the initial cost
functional of $Y_{0}$. In our formulation, if we take $l\left(  \omega
,0,X_{0},Y_{0},Z_{0},u_{0}\right)  =\gamma\left(  \omega,Y_{0}\right)  $, then
the cost functional (\ref{fbsde_cost_eq}) for our discrete time framework can
be reduced to the cost functional in \cite{p93} formally.
\end{remark}

For system (\ref{c5_eq_w_pc_state_eq})-(\ref{c5_eq_w_pc_cost_eq}), we assume that:

\begin{assumption}
\label{control_assumption}For $\varphi=b$, $\sigma_{i}$, $f$, $l$, $h$, we
assume that

\begin{enumerate}
\item $\varphi$\ is adapted map, i.e. for any\ $\left(  x,y,z,u\right)
\in\mathbb{R}^{m}\times\mathbb{R}^{n}\times\mathbb{\mathbb{R}}^{n\times
d}\times\mathbb{R}^{r}$, $\varphi\left(  \cdot,\cdot,x,y,z,u\right)  $ is
$\left\{  \mathcal{F}_{t}\right\}  $-adapted process. Moreover, $\varphi
\left(  \cdot,t,0,0,0,0\right)  \in L^{2}\left(  \mathcal{F}_{t}\right)  .$

\item $\forall t\in\left\{  0,1,...,T\right\}  $, $\varphi\left(
\cdot,t,\cdot,\cdot,\cdot,\cdot\right)  \,$is continuously differentiable with
respect to $x,y,z,u$, and $\varphi_{x},\varphi_{y},\varphi_{z_{i}},\varphi
_{u}$ are uniformly bounded $P-a.s.$. Also, for $t=T$, $f_{z_{i}}\equiv0$,
i.e. $f$ is independent of $z$ at time $T$. Here we use $z_{i}$ to represent
the $i$-th column of the matrix $z$.
\end{enumerate}
\end{assumption}

Let%
\[
\lambda=\left(
\begin{array}
[c]{c}%
x\\
y\\
z
\end{array}
\right)  ,A\left(  t,\lambda;u\right)  =\left(
\begin{array}
[c]{c}%
-f\\
b\\
\sigma
\end{array}
\right)  \left(  t,\lambda;u\right)  .
\]

For control system (\ref{fbsde_state_eq})-(\ref{fbsde_cost_eq}), we
additionally assume that:

\begin{assumption}
\label{control_assumption2}For any $u\in\mathcal{U}$, the coefficients in
(\ref{fbsde_state_eq}) satisfy the following monotone conditions, i.e. when
$t\in\left\{  1,...,T-1\right\}  $,%
\begin{gather*}
\left\langle A\left(  t,\lambda_{1};u\right)  -A\left(  t,\lambda
_{2};u\right)  ,\lambda_{1}-\lambda_{2}\right\rangle \leq-\alpha\left\vert
\lambda_{1}-\lambda_{2}\right\vert ^{2},P-a.s.,\\
\forall\lambda_{1},\lambda_{2}\in\mathbb{R}^{n}\times\mathbb{R}^{n}%
\times\mathbb{\mathbb{R}}^{n};
\end{gather*}

when $t=T$,%
\[
\left\langle -f\left(  T,x_{1},y,z,u\right)  +f\left(  T,x_{2},y,z,u\right)
,x_{1}-x_{2}\right\rangle \leq-\alpha\left\vert x_{1}-x_{2}\right\vert
^{2},P-a.s.;
\]

when $t=0$,%
\begin{align*}
&  \left\langle b\left(  0,\lambda_{1};u\right)  -b\left(  0,\lambda
_{2};u\right)  ,y_{1}-y_{2}\right\rangle +\left\langle \sigma\left(
0,\lambda_{1};u\right)  -\sigma\left(  0,\lambda_{2};u\right)  ,z_{1}%
-z_{2}\right\rangle \\
&  \leq-\alpha\left[  \left\vert y_{1}-y_{2}\right\vert ^{2}+\left\Vert
z_{1}-z_{2}\right\Vert ^{2}\right]  ,
\end{align*}
where $\alpha$ is a given positive constant.
\end{assumption}

Besides, in the following, we formally denote $b\left(  T,x,y,z,u\right)
\equiv0$, $\sigma\left(  T,x,y,z,u\right)  \equiv0$, $l\left(
T,x,y,z,u\right)  \equiv0$, $f\left(  0,x,y,z,u\right)  \equiv0$.

\section{Maximum principle for the partially coupled FBS$\Delta$E system}

For any $u\in\mathcal{U}$, it is obvious that there exists a unique solution
$\left\{  X_{t}\right\}  _{t=0}^{T}\in\mathcal{M}^{2}\left(  0,T;\mathbb{R}%
^{m}\right)  $ to the forward stochastic difference equation in the system
(\ref{c5_eq_w_pc_state_eq}). Then, by Theorem \ref{bsde_result_l}, the
backward equation in the system (\ref{c5_eq_w_pc_state_eq}) has a unique
solution $\left(  Y,Z,N\right)  $ where $Y=\left\{  Y_{t}\right\}  _{t=0}^{T}%
$, $Z=\left\{  Z_{t}\right\}  _{t=0}^{T-1}$ and $N=\left\{  N_{t}\right\}
_{t=0}^{T}$.

Suppose that $\bar{u}=\left\{  \bar{u}_{t}\right\}  _{t=0}^{T}$ is the optimal
control of problem (\ref{c5_eq_w_pc_state_eq})-(\ref{c5_eq_w_pc_cost_eq}) and
$\left(  \bar{X},\bar{Y},\bar{Z}\right)  $ is the corresponding optimal
trajectory. For a fixed time $0\leq s\leq T$, choose any $\Delta v\in
L^{2}\left(  \mathcal{F}_{s};\mathbb{R}^{r}\right)  $ such that $\bar{u}%
_{s}+\Delta v$ takes values in $U_{s}$. For any $\varepsilon\in\left[
0,1\right]  $, construct the perturbed admissible control
\begin{equation}
u_{t}^{\varepsilon}=\left(  1-\delta_{ts}\right)  \bar{u}_{t}+\delta
_{ts}\left(  \bar{u}_{s}+\varepsilon\Delta v\right)  =\bar{u}_{t}+\delta
_{ts}\varepsilon\Delta v, \label{perturb-control}%
\end{equation}
where $\delta_{ts}=1$ for $t=s$, $\delta_{ts}=0$ for $t\neq s$ and
$t\in\left\{  0,1,...,T\right\}  $. Since $U_{s}$ is a convex set, $\left\{
u_{t}^{\varepsilon}\right\}  _{t=0}^{T}\in\mathcal{U}$ is an admissible
control. Let $\left(  X^{\varepsilon},Y^{\varepsilon},Z^{\varepsilon
},N^{\varepsilon}\right)  $ be the solution of (\ref{c5_eq_w_pc_state_eq})
corresponding to the control $u^{\varepsilon}$.

Set%
\begin{equation}%
\begin{array}
[c]{rclrcl}%
\bar{\varphi}\left(  t\right)  & = & \varphi\left(  t,\bar{X}_{t},\bar{Y}%
_{t},\bar{Z}_{t},\bar{u}_{t}\right)  , & \varphi^{\varepsilon}\left(  t\right)
& = & \varphi\left(  t,X_{t}^{\varepsilon},Y_{t}^{\varepsilon},Z_{t}%
^{\varepsilon},u_{t}^{\varepsilon}\right)  ,\\
\widetilde{\varphi}^{\varepsilon}\left(  t\right)  & = & \varphi\left(
t,\bar{X}_{t},\bar{Y}_{t},\bar{Z}_{t},u_{t}^{\varepsilon}\right)  , &
\varphi_{\mu}\left(  t\right)  & = & \varphi_{\mu}\left(  t,\bar{X}_{t}%
,\bar{Y}_{t},\bar{Z}_{t},\bar{u}_{t}\right)  ,
\end{array}
\label{coefficient_notations}%
\end{equation}
where $\varphi=b$, $\sigma_{i}$, $g$, $f$, $l$, $h$ and $\mu=x$, $y$, $z_{i}$
and $u$.

Then, we have the following estimates.

\begin{lemma}
Under Assumption (\ref{control_assumption}), we have%
\begin{equation}
\sup_{0\leq t\leq T}\mathbb{E}\left\vert X_{t}^{\varepsilon}-\bar{X}%
_{t}\right\vert ^{2}\leq C\varepsilon^{2}\mathbb{E}\left\vert \Delta
v\right\vert ^{2}. \label{x_estimate_1_pc}%
\end{equation}

\end{lemma}

\begin{proof}
In the following, the positive constant $C$ may change from lines to lines.

When $t=0,...,s$, $X_{t}^{\varepsilon}=\bar{X}_{t}$.

When $t=s+1$,
\[
X_{s+1}^{\varepsilon}-\bar{X}_{s+1}=\widetilde{b}^{\varepsilon}\left(
s\right)  -\overline{b}\left(  s\right)  +\sum_{i=1}^{d}\left[
\widetilde{\sigma_{i}}^{\varepsilon}\left(  s\right)  -\overline{\sigma}%
_{i}\left(  s\right)  \right]  \Delta W_{s}^{i}.
\]
Then,%
\[
\mathbb{E}\left\vert X_{s+1}^{\varepsilon}-\bar{X}_{s+1}\right\vert ^{2}%
\leq\left(  d+1\right)  \mathbb{E}\left[  \left\vert \widetilde{b}%
^{\varepsilon}\left(  s\right)  -\overline{b}\left(  s\right)  \right\vert
^{2}+\sum_{i=1}^{d}\left\vert \left[  \widetilde{\sigma_{i}}^{\varepsilon
}\left(  s\right)  -\overline{\sigma}_{i}\left(  s\right)  \right]  \Delta
W_{s}^{i}\right\vert ^{2}\right]  .
\]
By the boundedness of $b_{u}$, we have
\[
\mathbb{E}\left[  \left\vert \widetilde{b}^{\varepsilon}\left(  s\right)
-\overline{b}\left(  s\right)  \right\vert ^{2}\right]  \leq C\mathbb{E}%
\left[  \left\vert u_{s}^{\varepsilon}-\bar{u}_{s}\right\vert ^{2}\right]
=C\varepsilon^{2}\mathbb{E}\left[  \left\vert \Delta v\right\vert ^{2}\right]
.
\]
By the boundedness of $\sigma_{iu}$, we have%
\begin{align*}
&  \mathbb{E}\left\vert \left[  \widetilde{\sigma_{i}}^{\varepsilon}\left(
s\right)  -\overline{\sigma}_{i}\left(  s\right)  \right]  \Delta W_{s}%
^{i}\right\vert ^{2}\\
&  =\mathbb{E}\left[  \left\vert \widetilde{\sigma_{i}}^{\varepsilon}\left(
s\right)  -\overline{\sigma}_{i}\left(  s\right)  \right\vert ^{2}%
|\mathbb{E}\left[  \left\vert \Delta W_{s}^{i}\right\vert ^{2}|\mathcal{F}%
_{s}\right]  \right]  \\
&  =\mathbb{E}\left[  \left\vert \left[  \widetilde{\sigma_{i}}^{\varepsilon
}\left(  s\right)  -\overline{\sigma}_{i}\left(  s\right)  \right]
\right\vert ^{2}\right]  \\
&  \leq C\varepsilon^{2}\mathbb{E}\left[  \left\vert \Delta v\right\vert
^{2}\right]
\end{align*}
which leads to%
\[
\mathbb{E}\left\vert X_{s+1}^{\varepsilon}-\bar{X}_{s+1}\right\vert ^{2}\leq
C\varepsilon^{2}\mathbb{E}\left[  \left\vert \Delta v\right\vert ^{2}\right]
.
\]

When $t=s+2,...,T$,%
\[%
\begin{array}
[c]{rl}%
\mathbb{E}\left\vert X_{t}^{\varepsilon}-\bar{X}_{t}\right\vert ^{2} &
\leq\left(  d+1\right)  \mathbb{E}\left\vert b\left(  t-1,X_{t-1}%
^{\varepsilon},\bar{u}_{t-1}\right)  -b\left(  t-1,\bar{X}_{t-1},\bar{u}%
_{t-1}\right)  \right\vert ^{2}\\
& +\sum_{i=1}^{d}\left\vert \left[  \sigma_{i}\left(  t-1,X_{t-1}%
^{\varepsilon},\bar{u}_{t-1}\right)  -\sigma_{i}\left(  t-1,\bar{X}_{t-1}%
,\bar{u}_{t-1}\right)  \right]  \Delta W_{s}^{i}\right\vert ^{2}.
\end{array}
\]
Due to the boundedness of $b_{x}$, $\sigma_{ix}$, we obtain $\mathbb{E}%
\left\vert X_{t}^{\varepsilon}-\bar{X}_{t}\right\vert ^{2}\leq C\mathbb{E}%
\left[  \left\vert X_{t-1}^{\varepsilon}-\bar{X}_{t-1}\right\vert ^{2}\right]
$. Thus, by induction we prove the result.
\end{proof}

Let $\xi=\left\{  \xi_{t}\right\}  _{t=0}^{T}$ be the solution to the
following difference equation,%
\begin{equation}
\left\{
\begin{array}
[c]{rcl}%
\Delta\xi_{t} & = & b_{x}\left(  t\right)  \xi_{t}+\delta_{ts}b_{u}\left(
t\right)  \varepsilon\Delta v+\sum_{i=1}^{d}\left[  \sigma_{ix}\left(
t\right)  \xi_{t}+\delta_{ts}\varepsilon\sigma_{iu}\left(  t\right)  \Delta
v\right]  \Delta W_{t}^{i},\\
\xi_{0} & = & 0,
\end{array}
\right.  \label{variational eq-xi}%
\end{equation}

It is easy to check that%
\begin{equation}
\sup_{0\leq t\leq T}\mathbb{E}\left\vert \xi_{t}\right\vert ^{2}\leq
C\varepsilon^{2}\mathbb{E}\left\vert \Delta v\right\vert ^{2},
\label{xi_estimate}%
\end{equation}

and we have the following result:

\begin{lemma}
\label{x_approxi}Under Assumption \ref{control_assumption}, we have%
\[
\sup_{0\leq t\leq T}\mathbb{E}\left\vert X_{t}^{\varepsilon}-\bar{X}_{t}%
-\xi_{t}\right\vert ^{2}=o\left(  \varepsilon^{2}\right)  .
\]

\end{lemma}

\begin{proof}
When $t=0,...,s$, $X_{t}^{\varepsilon}=\bar{X}_{t}$ and $\xi_{t}=0$ which lead
to $X_{t}^{\varepsilon}-\bar{X}_{t}-\xi_{t}=0.$

When $t=s+1$,%
\[
X_{s+1}^{\varepsilon}-\bar{X}_{s+1}-\xi_{s+1}=\left[  \widetilde{b}_{u}\left(
s\right)  -b_{u}\left(  s\right)  \right]  \varepsilon\Delta v+\sum_{i=1}%
^{d}\left[  \widetilde{\sigma}_{iu}\left(  s\right)  -\sigma_{iu}\left(
s\right)  \right]  \varepsilon\Delta v\Delta W_{s}^{i}%
\]
where
\begin{align*}
\widetilde{b}_{u}\left(  s\right)   &  =\int_{0}^{1}b_{u}\left(  s,\bar{X}%
_{s},\bar{u}_{s}+\lambda\left(  u_{s}^{\varepsilon}-\bar{u}_{s}\right)
\right)  d\lambda,\\
\widetilde{\sigma}_{iu}\left(  s\right)   &  =\int_{0}^{1}\sigma_{iu}\left(
s,\bar{X}_{s},\bar{u}_{s}+\lambda\left(  u_{s}^{\varepsilon}-\bar{u}%
_{s}\right)  \right)  d\lambda.
\end{align*}
Then
\[%
\begin{array}
[c]{rl}%
\mathbb{E}\left\vert X_{s+1}^{\varepsilon}-\bar{X}_{s+1}-\xi_{s+1}\right\vert
^{2} & \leq\left(  d+1\right)  \mathbb{E}\left[  \left\vert \left[
\widetilde{b}_{u}\left(  s\right)  -b_{u}\left(  s\right)  \right]
\varepsilon\Delta v\right\vert ^{2}+\sum_{i=1}^{d}\left\vert \left[
\widetilde{\sigma}_{iu}\left(  s\right)  -\sigma_{iu}\left(  s\right)
\right]  \varepsilon\Delta v\Delta W_{s}^{i}\right\vert ^{2}\right] \\
& \leq C\mathbb{E}\left[  \left\Vert \widetilde{b}_{u}\left(  s\right)
-b_{u}\left(  s\right)  \right\Vert ^{2}\left\vert \Delta v\right\vert
^{2}+\sum_{i=1}^{d}\left\Vert \widetilde{\sigma}_{iu}\left(  s\right)
-\sigma_{iu}\left(  s\right)  \right\Vert ^{2}\left\vert \Delta v\right\vert
^{2}\right]  \varepsilon^{2}.
\end{array}
\]
Since $\left\Vert \widetilde{b}_{u}\left(  s\right)  -b_{u}\left(  s\right)
\right\Vert \rightarrow0$ and $\left\Vert \widetilde{\sigma}_{iu}\left(
s\right)  -\sigma_{iu}\left(  s\right)  \right\Vert \rightarrow0$ as
$\varepsilon\rightarrow0$, we have%
\[
\lim_{\varepsilon\rightarrow0}\frac{1}{\varepsilon^{2}}\mathbb{E}\left\vert
X_{s+1}^{\varepsilon}-\bar{X}_{s+1}-\xi_{s+1}\right\vert ^{2}=0.
\]

When $t=s+2,...,T$,%
\[%
\begin{array}
[c]{rl}%
X_{t}^{\varepsilon}-\bar{X}_{t}-\xi_{t} & =\widetilde{b}_{x}\left(
t-1\right)  \left(  X_{t-1}^{\varepsilon}-\bar{X}_{t-1}-\xi_{t-1}\right)
+\left[  \widetilde{b}_{x}\left(  t-1\right)  -b_{x}\left(  t-1\right)
\right]  \xi_{t-1}\\
& +\sum_{i=1}^{d}\left\{  \widetilde{\sigma}_{ix}\left(  t-1\right)  \left(
X_{t-1}^{\varepsilon}-\bar{X}_{t-1}-\xi_{t-1}\right)  +\left[
\widetilde{\sigma}_{ix}\left(  t-1\right)  -\sigma_{ix}\left(  t-1\right)
\right]  \xi_{t-1}\right\}  \Delta W_{t-1}^{i}%
\end{array}
\]
where
\begin{align*}
\widetilde{b}_{x}\left(  t\right)   &  =\int_{0}^{1}b_{x}\left(  t,\bar{X}%
_{t}+\lambda\left(  X_{t}^{\varepsilon}-\bar{X}_{t}\right)  ,\bar{u}%
_{t}\right)  d\lambda,\\
\widetilde{\sigma}_{ix}\left(  t\right)   &  =\int_{0}^{1}\sigma_{ix}\left(
t,\bar{X}_{t}+\lambda\left(  X_{t}^{\varepsilon}-\bar{X}_{t}\right)  ,\bar
{u}_{t}\right)  d\lambda.
\end{align*}
Then%
\[%
\begin{array}
[c]{rl}%
\mathbb{E}\left\vert X_{t}^{\varepsilon}-\bar{X}_{t}-\xi_{t}\right\vert
^{2}\leq & C\mathbb{E}\left\Vert \widetilde{b}_{x}\left(  t-1\right)
\right\Vert ^{2}\left\vert X_{t-1}^{\varepsilon}-\bar{X}_{t-1}-\xi
_{t-1}\right\vert ^{2}+\left\Vert \widetilde{b}_{x}\left(  t-1\right)
-b_{x}\left(  t-1\right)  \right\Vert ^{2}\left\vert \xi_{t-1}\right\vert
^{2}\\
& +\sum_{i=1}^{d}\left\Vert \widetilde{\sigma}_{ix}\left(  t-1\right)
\right\Vert ^{2}\left\vert X_{t-1}^{\varepsilon}-\bar{X}_{t-1}-\xi
_{t-1}\right\vert ^{2}+\sum_{i=1}^{d}\left\Vert \widetilde{\sigma}_{ix}\left(
t-1\right)  -\sigma_{ix}\left(  t-1\right)  \right\Vert ^{2}\left\vert
\xi_{t-1}\right\vert ^{2}\\
\leq & C\mathbb{E}\left(  \left\Vert \widetilde{b}_{x}\left(  t-1\right)
\right\Vert ^{2}+\sum_{i=1}^{d}\left\Vert \widetilde{\sigma}_{ix}\left(
t-1\right)  \right\Vert ^{2}\right)  \left\vert X_{t-1}^{\varepsilon}-\bar
{X}_{t-1}-\xi_{t-1}\right\vert ^{2}\\
& +\left\Vert \widetilde{b}_{x}\left(  t-1\right)  -b_{x}\left(  t-1\right)
\right\Vert ^{2}\left\vert \xi_{t-1}\right\vert ^{2}\left.  +\sum_{i=1}%
^{d}\left\Vert \widetilde{\sigma}_{ix}\left(  t-1\right)  -\sigma_{ix}\left(
t-1\right)  \right\Vert ^{2}\left\vert \xi_{t-1}\right\vert ^{2}\right]  .
\end{array}
\]
$\left\Vert \widetilde{b}_{x}\left(  t-1\right)  -b_{x}\left(  t-1\right)
\right\Vert \rightarrow0$ and $\left\Vert \widetilde{\sigma}_{ix}\left(
t-1\right)  -\sigma_{ix}\left(  t-1\right)  \right\Vert \rightarrow0$ as
$\varepsilon\rightarrow0$. Since $\widetilde{b}_{x}\left(  t-1\right)  $ and
$\widetilde{\sigma}_{ix}\left(  t-1\right)  $ are bounded, by the estimation
(\ref{xi_estimate}), we have%
\[
\lim_{\varepsilon\rightarrow0}\frac{1}{\varepsilon^{2}}\mathbb{E}\left\vert
X_{t}^{\varepsilon}-\bar{X}_{t}-\xi_{t}\right\vert ^{2}=0.
\]
This completes the proof.
\end{proof}

\begin{lemma}
Under Assumption \ref{control_assumption}, we have
\begin{align}
\sup_{0\leq t\leq T}\mathbb{E}\left\vert Y_{t}^{\varepsilon}-\bar{Y}%
_{t}\right\vert ^{2}  &  \leq C\varepsilon^{2}\mathbb{E}\left\vert \Delta
v\right\vert ^{2}\label{y_estimate_1}\\
\sup_{0\leq t\leq T-1}\mathbb{E}\left\Vert Z_{t}^{\varepsilon}-\bar{Z}%
_{t}\right\Vert ^{2}  &  \leq C\varepsilon^{2}\mathbb{E}\left\vert \Delta
v\right\vert ^{2}. \label{z_estimate_1}%
\end{align}

\end{lemma}

\begin{proof}
It is obvious that $Y_{T}^{\varepsilon}-\bar{Y}_{T}=0$ at time $T$.

When $t=s,...,T-1$ (if $s=T$, skip this part), we have%
\begin{align*}
&  \mathbb{E}\left\vert f\left(  t+1,X_{t+1}^{\varepsilon},Y_{t+1}%
^{\varepsilon},Z_{t+1}^{\varepsilon},\bar{u}_{t+1}\right)  -f\left(
t+1,\bar{X}_{t+1},\bar{Y}_{t+1},\bar{Z}_{t+1},\bar{u}_{t+1}\right)
\right\vert ^{2}\\
&  \leq C\mathbb{E}\left[  \left\vert X_{t+1}^{\varepsilon}-\bar{X}%
_{t+1}\right\vert ^{2}+\left\vert Y_{t+1}^{\varepsilon}-\bar{Y}_{t+1}%
\right\vert ^{2}+\left\Vert Z_{t+1}^{\varepsilon}-\bar{Z}_{t+1}\right\Vert
^{2}\right] \\
&  \leq C\mathbb{E}\left[  \left\vert Y_{t+1}^{\varepsilon}-\bar{Y}%
_{t+1}\right\vert ^{2}+\left\Vert Z_{t+1}^{\varepsilon}-\bar{Z}_{t+1}%
\right\Vert ^{2}\right]  +C_{1}\varepsilon^{2}\mathbb{E}\left\vert \Delta
v\right\vert ^{2}.
\end{align*}
It yields that%
\[
\mathbb{E}\left\vert Y_{t}^{\varepsilon}-\bar{Y}_{t}\right\vert ^{2}\leq
C\mathbb{E}\left[  \left\vert Y_{t+1}^{\varepsilon}-\bar{Y}_{t+1}\right\vert
^{2}+\left\Vert Z_{t+1}^{\varepsilon}-\bar{Z}_{t+1}\right\Vert ^{2}\right]
+C\varepsilon^{2}\mathbb{E}\left\vert \Delta v\right\vert ^{2}.
\]
Similarly, we have%
\begin{align*}
&  \mathbb{E}\left\Vert Z_{t}^{\varepsilon}-\bar{Z}_{t}\right\Vert ^{2}\\
&  \leq C\mathbb{E}\left[  \left\vert Y_{t+1}^{\varepsilon}-\bar{Y}%
_{t+1}\right\vert ^{2}+\left\Vert Z_{t+1}^{\varepsilon}-\bar{Z}_{t+1}%
\right\Vert ^{2}\right]  +C\varepsilon^{2}\mathbb{E}\left\vert \Delta
v\right\vert ^{2}.
\end{align*}

When $t=s-1$, by similar analysis,%
\[
\left\{
\begin{array}
[c]{rcc}%
\mathbb{E}\left\vert Y_{t}^{\varepsilon}-\bar{Y}_{t}\right\vert ^{2} & \leq &
C\mathbb{E}\left[  \left\vert Y_{t+1}^{\varepsilon}-\bar{Y}_{t+1}\right\vert
^{2}+\left\Vert Z_{t+1}^{\varepsilon}-\bar{Z}_{t+1}\right\Vert ^{2}\right]
+C\varepsilon^{2}\mathbb{E}\left\vert \Delta v\right\vert ^{2},\\
\mathbb{E}\left[  \left\Vert Z_{t}^{\varepsilon}-\bar{Z}_{t}\right\Vert
^{2}\right]  & \leq & C\mathbb{E}\left[  \left\vert Y_{t+1}^{\varepsilon}%
-\bar{Y}_{t+1}\right\vert ^{2}+\left\Vert Z_{t+1}^{\varepsilon}-\bar{Z}%
_{t+1}\right\Vert ^{2}\right]  +C\varepsilon^{2}\mathbb{E}\left\vert \Delta
v\right\vert ^{2}.
\end{array}
\right.
\]
If $s=T$, it shows like%
\[
\left\{
\begin{array}
[c]{rcc}%
\mathbb{E}\left\vert Y_{T-1}^{\varepsilon}-\bar{Y}_{T-1}\right\vert ^{2} &
\leq & C\varepsilon^{2}\mathbb{E}\left\vert \Delta v\right\vert ^{2},\\
\mathbb{E}\left[  \left\Vert Z_{T-1}^{\varepsilon}-\bar{Z}_{T-1}\right\Vert
^{2}\right]  & \leq & C\varepsilon^{2}\mathbb{E}\left\vert \Delta v\right\vert
^{2}.
\end{array}
\right.
\]

When $t=0,...,s-2$, we have
\[
\left\{
\begin{array}
[c]{rcc}%
\mathbb{E}\left\vert Y_{t}^{\varepsilon}-\bar{Y}_{t}\right\vert ^{2} & \leq &
C\mathbb{E}\left[  \left\vert Y_{t+1}^{\varepsilon}-\bar{Y}_{t+1}\right\vert
^{2}+\left\Vert Z_{t+1}^{\varepsilon}-\bar{Z}_{t+1}\right\Vert ^{2}\right]
,\\
\mathbb{E}\left[  \left\Vert Z_{t}^{\varepsilon}-\bar{Z}_{t}\right\Vert
^{2}\right]  & \leq & C\mathbb{E}\left[  \left\vert Y_{t+1}^{\varepsilon}%
-\bar{Y}_{t+1}\right\vert ^{2}+\left\Vert Z_{t+1}^{\varepsilon}-\bar{Z}%
_{t+1}\right\Vert ^{2}\right]  .
\end{array}
\right.
\]

Thus, there exists $C>0$, such that for any $t\in\left\{  0,1,...,T\right\}
$,
\[
\left\{
\begin{array}
[c]{rcc}%
\mathbb{E}\left\vert Y_{t}^{\varepsilon}-\bar{Y}_{t}\right\vert ^{2} & \leq &
C\varepsilon^{2}\mathbb{E}\left\vert \Delta v\right\vert ^{2},\\
\mathbb{E}\left[  \left\Vert Z_{t}^{\varepsilon}-\bar{Z}_{t}\right\Vert
^{2}\right]  & \leq & C\varepsilon^{2}\mathbb{E}\left\vert \Delta v\right\vert
^{2}.
\end{array}
\right.
\]
This completes the proof.
\end{proof}

Let $\left(  \eta,\zeta,V\right)  $ be the solution to the following
BS$\Delta$E,%

\[
\left\{
\begin{array}
[c]{rcl}%
\Delta\eta_{t} & = & -f_{x}\left(  t+1\right)  \xi_{t+1}-f_{y}\left(
t+1\right)  \eta_{t+1}-\delta_{\left(  t+1\right)  s}f_{u}\left(  t+1\right)
\varepsilon\Delta v\\
&  & -\sum_{i=1}^{d}f_{z_{i}}\left(  t+1\right)  \zeta_{t+1}e_{i}+\zeta
_{t}\Delta W_{t}+\Delta V_{t},\\
\eta_{T} & = & 0.
\end{array}
\right.
\]
Notice that $f_{x}\left(  T\right)  =f_{x}\left(  T,\bar{X}_{T},\bar{Y}%
_{T},\bar{u}_{T}\right)  $ since $f$ is independent of $Z$, also as
$f_{y}\left(  T\right)  $, $f_{u}\left(  T\right)  $.

It is easy to check that%
\begin{align*}
\sup_{0\leq t\leq T}\mathbb{E}\left\vert \eta_{t}\right\vert ^{2}  &  \leq
C\varepsilon^{2}\mathbb{E}\left\vert \Delta v\right\vert ^{2},\\
\sup_{0\leq t\leq T-1}\mathbb{E}\left\Vert \zeta_{t}\right\Vert ^{2}  &  \leq
C\varepsilon^{2}\mathbb{E}\left\vert \Delta v\right\vert ^{2}.
\end{align*}

and we have the following result:

\begin{lemma}
\label{yz_approxi}Under Assumption \ref{control_assumption}, we have
\begin{align*}
\sup_{0\leq t\leq T}\mathbb{E}\left\vert Y_{t}^{\varepsilon}-\bar{Y}_{t}%
-\eta_{t}\right\vert ^{2}  &  =o\left(  \varepsilon^{2}\right)  ,\\
\sup_{0\leq t\leq T-1}\mathbb{E}\left\Vert Z_{t}^{\varepsilon}-\bar{Z}%
_{t}-\zeta_{t}\right\Vert ^{2}  &  =o\left(  \varepsilon^{2}\right)  .
\end{align*}

\end{lemma}

\begin{proof}
When $t=T$, $Y_{T}^{\varepsilon}-\bar{Y}_{T}-\eta_{T}=0$.

When $t\in\left\{  0,1,...,T-1\right\}  $, we have%
\begin{align*}
&
\begin{array}
[c]{cc}
& Y_{t}^{\varepsilon}-\bar{Y}_{t}-\eta_{t}%
\end{array}
\\
&
\begin{array}
[c]{cl}%
= & \mathbb{E}\left[  Y_{t+1}^{\varepsilon}+f^{\varepsilon}\left(  t+1\right)
-\bar{Y}_{t+1}-\overline{f}\left(  t+1\right)  -\eta_{t+1}-f_{x}\left(
t+1\right)  \xi_{t+1}\right. \\
& \left.  -f_{y}\left(  t+1\right)  \eta_{t+1}-\sum_{i=1}^{d}f_{z_{i}}\left(
t+1\right)  \zeta_{t+1}e_{i}-\delta_{\left(  t+1\right)  s}f_{u}\left(
t+1\right)  \varepsilon\Delta v|\mathcal{F}_{t}\right]
\end{array}
\\
&
\begin{array}
[c]{cr}%
= & \mathbb{E}\left[  Y_{t+1}^{\varepsilon}-\bar{Y}_{t+1}-\eta_{t+1}%
+\widetilde{f}_{x}\left(  t+1\right)  \left(  X_{t+1}^{\varepsilon}-\bar
{X}_{t+1}\right)  +\widetilde{f}_{y}\left(  t+1\right)  \left(  Y_{t+1}%
^{\varepsilon}-\bar{Y}_{t+1}\right)  \right. \\
& +\sum_{i=1}^{d}\widetilde{f}_{z_{i}}\left(  t+1\right)  \left(
Z_{t+1}^{\varepsilon}-\bar{Z}_{t+1}\right)  e_{i}+\delta_{\left(  t+1\right)
s}\widetilde{f}_{u}\left(  t+1\right)  \varepsilon\Delta v-f_{x}\left(
t+1\right)  \xi_{t+1}\\
& \left.  -f_{y}\left(  t+1\right)  \eta_{t+1}-\sum_{i=1}^{d}f_{z_{i}}\left(
t+1\right)  \zeta_{t+1}e_{i}-\delta_{\left(  t+1\right)  s}f_{u}\left(
t+1\right)  \varepsilon\Delta v|\mathcal{F}_{t}\right]  ,
\end{array}
\end{align*}
where%
\[
\widetilde{f}_{\mu}\left(  t\right)  =\int_{0}^{1}f_{\mu}(t,\bar{X}%
_{t}+\lambda\left(  X_{t}^{\varepsilon}-\bar{X}_{t}\right)  ,\bar{Y}%
_{t}+\lambda\left(  Y_{t}^{\varepsilon}-\bar{Y}_{t}\right)  ,\bar{Z}%
_{t}+\lambda\left(  Z_{t}^{\varepsilon}-\bar{Z}_{t}\right)  ,\bar{u}%
_{t}+\lambda\left(  u_{t}^{\varepsilon}-\bar{u}_{t}\right)  d\lambda
\]
for $\mu=x$, $y$, $z_{i}$ and $u$. Then,
\begin{align*}
&  \mathbb{E}\left\vert Y_{t}^{\varepsilon}-\bar{Y}_{t}-\eta_{t}\right\vert
^{2}\\
&  \leq C\mathbb{E}\left[  \left\vert Y_{t+1}^{\varepsilon}-\bar{Y}_{t+1}%
-\eta_{t+1}\right\vert ^{2}\right. \\
&  +\left\vert \widetilde{f}_{x}\left(  t+1\right)  \left(  X_{t+1}%
^{\varepsilon}-\bar{X}_{t+1}-\xi_{t+1}\right)  \right\vert ^{2}+\left\vert
\left[  \widetilde{f}_{x}\left(  t+1\right)  -f_{x}\left(  t+1\right)
\right]  \xi_{t+1}\right\vert ^{2}\\
&  +\left\vert \widetilde{f}_{y}\left(  t+1\right)  \left(  Y_{t+1}%
^{\varepsilon}-\bar{Y}_{t+1}-\eta_{t+1}\right)  \right\vert ^{2}+\left\vert
\left[  \widetilde{f}_{y}\left(  t+1\right)  -f_{y}\left(  t+1\right)
\right]  \eta_{t+1}\right\vert ^{2}\\
&  +\sum_{i=1}^{d}\left\vert \widetilde{f}_{z_{i}}\left(  t+1\right)  \left(
Z_{t+1}^{\varepsilon}-\bar{Z}_{t+1}-\zeta_{t+1}\right)  e_{i}\right\vert
^{2}+\sum_{i=1}^{d}\left\vert \left[  \widetilde{f}_{z_{i}}\left(  t+1\right)
-f_{z_{i}}\left(  t+1\right)  \right]  \zeta_{t+1}e_{i}\right\vert ^{2}\\
&  \left.  +\delta_{\left(  t+1\right)  s}\left\vert \left[  \widetilde{f}%
_{u}\left(  t+1\right)  -f_{u}\left(  t+1\right)  \right]  \varepsilon\Delta
v\right\vert ^{2}\right]
\end{align*}
and%
\begin{align*}
&  \mathbb{E}\left\Vert Z_{t}^{\varepsilon}-\bar{Z}_{t}-\zeta_{t}\right\Vert
^{2}\\
&  \leq C\mathbb{E}\left[  \left\vert Y_{t+1}^{\varepsilon}-\bar{Y}_{t+1}%
-\eta_{t+1}\right\vert ^{2}\right. \\
&  +\left\vert \widetilde{f}_{x}\left(  t+1\right)  \left(  X_{t+1}%
^{\varepsilon}-\bar{X}_{t+1}-\xi_{t+1}\right)  \right\vert ^{2}+\left\vert
\left[  \widetilde{f}_{x}\left(  t+1\right)  -f_{x}\left(  t+1\right)
\right]  \xi_{t+1}\right\vert ^{2}\\
&  +\left\vert \widetilde{f}_{y}\left(  t+1\right)  \left(  Y_{t+1}%
^{\varepsilon}-\bar{Y}_{t+1}-\eta_{t+1}\right)  \right\vert ^{2}+\left\vert
\left[  \widetilde{f}_{y}\left(  t+1\right)  -f_{y}\left(  t+1\right)
\right]  \eta_{t+1}\right\vert ^{2}\\
&  +\sum_{i=1}^{d}\left\vert \widetilde{f}_{z_{i}}\left(  t+1\right)  \left(
Z_{t+1}^{\varepsilon}-\bar{Z}_{t+1}-\zeta_{t+1}\right)  e_{i}\right\vert
^{2}+\sum_{i=1}^{d}\left\vert \left[  \widetilde{f}_{z_{i}}\left(  t+1\right)
-f_{z_{i}}\left(  t+1\right)  \right]  \zeta_{t+1}e_{i}\right\vert ^{2}\\
&  \left.  +\delta_{\left(  t+1\right)  s}\left\vert \left[  \widetilde{f}%
_{u}\left(  t+1\right)  -f_{u}\left(  t+1\right)  \right]  \varepsilon\Delta
v\right\vert ^{2}\right]  .
\end{align*}
Notice that $\widetilde{f}_{x}\left(  t\right)  -f_{x}\left(  t\right)
\rightarrow0,$ $\widetilde{f}_{y}\left(  t\right)  -f_{y}\left(  t\right)
\rightarrow0,$ $\widetilde{f}_{z_{i}}\left(  t\right)  -f_{z_{i}}\left(
t\right)  \rightarrow0,$ $\widetilde{f}_{u}\left(  t\right)  -f_{u}\left(
t\right)  \rightarrow0$ as $\varepsilon\rightarrow0$. We obtain that%
\[%
\begin{array}
[c]{ccc}%
\lim_{\varepsilon\rightarrow0}\frac{1}{\varepsilon^{2}}\mathbb{E}\left\vert
Y_{t}^{\varepsilon}-\bar{Y}_{t}-\eta_{t}\right\vert ^{2} & = & 0,\\
\lim_{\varepsilon\rightarrow0}\frac{1}{\varepsilon^{2}}\mathbb{E}\left\Vert
Z_{t}^{\varepsilon}-\bar{Z}_{t}-\zeta_{t}\right\Vert ^{2} & = & 0.
\end{array}
\]
This completes the proof.
\end{proof}

By Lemma \ref{x_approxi} and Lemma \ref{yz_approxi}, we have%
\begin{align*}
&
\begin{array}
[c]{cc}
& J\left(  u^{\varepsilon}\left(  \cdot\right)  \right)  -J\left(  \bar
{u}\left(  \cdot\right)  \right)
\end{array}
\\
&
\begin{array}
[c]{cc}%
= & \mathbb{E}\sum_{t=0}^{T-1}\left[  \left\langle l_{x}\left(  t\right)
,\xi_{t}\right\rangle +\left\langle l_{y}\left(  t\right)  ,\eta
_{t}\right\rangle +\sum_{i=1}^{d}\left\langle l_{z_{i}}\left(  t\right)
,\zeta_{t}e_{i}\right\rangle +\delta_{ts}\left\langle l_{u}\left(  s\right)
,\varepsilon\Delta v\right\rangle \right] \\
& +\mathbb{E}\left\langle h_{x}\left(  \bar{X}_{T}\right)  ,\xi_{T}%
\right\rangle +o\left(  \varepsilon\right)  .
\end{array}
\end{align*}

Introducing the following adjoint equation:%

\begin{equation}
\left\{
\begin{array}
[c]{rcl}%
\Delta p_{t} & = & -b_{x}^{\ast}\left(  t+1\right)  p_{t+1}-\sum_{i=1}%
^{d}\sigma_{ix}^{\ast}\left(  t+1\right)  q_{t+1}e_{i}\\
&  & +f_{x}^{\ast}\left(  t+1\right)  k_{t+1}+l_{x}\left(  t+1\right)
+q_{t}\Delta W_{t}+\Delta Q_{t},\\
\Delta k_{t} & = & f_{y}^{\ast}\left(  t\right)  k_{t}+l_{y}\left(  t\right)
+\sum_{i=1}^{d}\left[  f_{z_{i}}^{\ast}\left(  t\right)  k_{t}+l_{z_{i}%
}\left(  t\right)  \right]  \Delta W_{t}^{i},\\
p_{T} & = & -h_{x}\left(  \bar{X}_{T}\right)  ,\\
k_{0} & = & 0,
\end{array}
\right.  \label{adjoint_eq_pc}%
\end{equation}
where $W$ and $Q$ are square integrable martingale processes and $Q$ is
strongly orthogonal to $W$.

Obviously the forward equation in (\ref{adjoint_eq_pc}) admits a unique
solution $k\in\mathcal{M}^{2}\left(  0,T;\mathbb{R}^{n}\right)  $. Then, based
on the solution $k$, according to Theorem \ref{bsde_result_l}, the backward
equation in (\ref{adjoint_eq_pc}) has a unique solution $\left(  p,q,Q\right)
\in\mathcal{M}^{2}\left(  0,T;\mathbb{R}^{m}\right)  \times\mathcal{M}%
^{2}\left(  0,T-1;\mathbb{R}^{m\times d}\right)  \times\mathcal{M}^{2}\left(
0,T;\mathbb{R}^{m}\right)  $. So FBS$\Delta$E has a unique solution $\left(
p,q,Q,k\right)  $.

We obtain the following maximum principle for the optimal control problem
(\ref{c5_eq_w_pc_state_eq})-(\ref{c5_eq_w_pc_cost_eq}).

Define the Hamiltonian function%
\[%
\begin{array}
[c]{r}%
H\left(  \omega,t,u,x,y,z,p,q,k\right)  =b^{\ast}\left(  \omega,t,x,u\right)
p+\sum_{i=1}^{d}\sigma_{i}^{\ast}\left(  \omega,t,x,u\right)  qe_{i}\\
-f^{\ast}\left(  \omega,t,x,y,z,u\right)  k-l\left(  \omega,t,x,y,z,u\right)
.
\end{array}
\]

\begin{theorem}
Suppose that Assumption (\ref{control_assumption}) holds. Let $\bar{u}$ be an
optimal control of the problem (\ref{c5_eq_w_pc_state_eq}%
)-(\ref{c5_eq_w_pc_cost_eq}), $\left(  \bar{X},\bar{Y},\bar{Z}\right)  $ be
the corresponding optimal trajectory and $\left(  p,q,k\right)  $ be the
solution to the adjoint equation (\ref{adjoint_eq_pc}). Then for any
$t\in\left\{  0,1,...,T\right\}  $, for any $v\in U_{t}$, we have%
\begin{equation}
\left\langle H_{u}\left(  t,\bar{u}_{t},\bar{X}_{t},\bar{Y}_{t},\bar{Z}%
_{t},p_{t},q_{t},k_{t}\right)  ,v-\bar{u}_{t}\right\rangle \leq0,\text{
}P-a.s.. \label{MP_result_pc}%
\end{equation}

\end{theorem}

\begin{proof}
For $t\in\left\{  0,1,...,T-1\right\}  $, we have%
\begin{equation}%
\begin{array}
[c]{rl}
& \Delta\left\langle \xi_{t},p_{t}\right\rangle \\
= & \left\langle \xi_{t+1},\Delta p_{t}\right\rangle +\left\langle \Delta
\xi_{t},p_{t}\right\rangle \\
= & \left\langle \xi_{t+1},-b_{x}^{\ast}\left(  t+1\right)  p_{t+1}-\sum
_{i=1}^{d}\sigma_{ix}^{\ast}\left(  t+1\right)  q_{t+1}e_{i}+f_{x}^{\ast
}\left(  t+1\right)  k_{t+1}+l_{x}\left(  t+1\right)  \right\rangle \\
& +\left\langle \sum_{j=1}^{d}\left[  \sigma_{jx}\left(  t\right)  \xi
_{t}+\delta_{ts}\varepsilon\sigma_{ju}\left(  t\right)  \Delta v\right]
\Delta W_{t}^{j},q_{t}\Delta W_{t}\right\rangle +\left\langle b_{x}\left(
t\right)  \xi_{t}+\delta_{ts}b_{u}\left(  t\right)  \varepsilon\Delta
v,p_{t}\right\rangle +\Phi_{t},
\end{array}
\label{product-rule-1}%
\end{equation}
where%
\[%
\begin{array}
[c]{ccl}%
\Phi_{t} & = & \left\langle \xi_{t}+b_{x}\left(  t\right)  \xi_{t}+\delta
_{ts}b_{u}\left(  t\right)  \varepsilon\Delta v,q_{t}\Delta W_{t}\right\rangle
+\left\langle \sum_{i=1}^{d}\left[  \sigma_{ix}\left(  t\right)  \xi
_{t}+\delta_{ts}\varepsilon\sigma_{iu}\left(  t\right)  \Delta v\right]
\Delta W_{t}^{i},p_{t}\right\rangle \\
&  & +\left\langle \xi_{t}+b_{x}\left(  t\right)  \xi_{t}+\delta_{ts}%
b_{u}\left(  t\right)  \varepsilon\Delta v,\Delta Q_{t}\right\rangle
+\left\langle \sum_{j=1}^{d}\left[  \sigma_{jx}\left(  t\right)  \xi
_{t}+\delta_{ts}\varepsilon\sigma_{ju}\left(  t\right)  \Delta v\right]
\Delta W_{t}^{j},\Delta Q_{t}\right\rangle .
\end{array}
\]
It is obvious that $\mathbb{E}\left[  \Phi_{t}\right]  =0$. We have%
\[%
\begin{array}
[c]{rl}%
\mathbb{E}\left\langle \sum_{j=1}^{d}\sigma_{jx}\left(  t\right)  \xi
_{t}\Delta W_{t}^{j},q_{t}\Delta W_{t}\right\rangle = & \mathbb{E}\left\langle
\sum_{j=1}^{d}\sigma_{jx}\left(  t\right)  \xi_{t}\Delta W_{t}^{j},\sum
_{i=1}^{d}q_{t}e_{i}\Delta W_{t}^{i}\right\rangle \\
= & \mathbb{E}\left[  \sum_{i=1}^{d}\left\langle \xi_{t},\sum_{j=1}^{d}%
\sigma_{jx}^{\ast}\left(  t\right)  q_{t}e_{i}\mathbb{E}\left[  \Delta
W_{t}^{i}\Delta W_{t}^{j}|\mathcal{F}_{t}\right]  \right\rangle \right] \\
= & \mathbb{E}\left[  \sum_{i=1}^{d}\left\langle \xi_{t},\sigma_{ix}^{\ast
}\left(  t\right)  q_{t}e_{i}\right\rangle \right]  ,
\end{array}
\]
and%
\[
\mathbb{E}\left\langle \sum_{j=1}^{d}\delta_{ts}\varepsilon\sigma_{ju}\left(
t\right)  \Delta v\Delta W_{t}^{j},q_{t}\Delta W_{t}\right\rangle
=\mathbb{E}\left[  \delta_{ts}\varepsilon\sum_{j=1}^{d}\left\langle
\sigma_{ju}\left(  t\right)  \Delta v,q_{t}e_{j}\right\rangle \right]  .
\]

Similarly, it can be shown that for $t\in\left\{  0,1,...,T-1\right\}  $, we
have%
\[%
\begin{array}
[c]{rl}
& \Delta\left\langle \eta_{t},k_{t}\right\rangle \\
= & \left\langle -f_{x}\left(  t+1\right)  \xi_{t+1}-f_{y}\left(  t+1\right)
\eta_{t+1}-\sum_{i=1}^{d}f_{z_{i}}\left(  t+1\right)  \zeta_{t+1}e_{i}%
-\delta_{\left(  t+1\right)  s}f_{u}\left(  t+1\right)  \varepsilon\Delta
v,k_{t+1}\right\rangle \\
& +\left\langle \zeta_{t}\Delta W_{t},\sum_{i=1}^{d}\left[  f_{z_{i}}^{\ast
}\left(  t\right)  k_{t}+l_{z_{i}}\left(  t\right)  \right]  \Delta W_{t}%
^{i}\right\rangle +\left\langle \eta_{t},f_{y}^{\ast}\left(  t\right)
k_{t}+l_{y}\left(  t\right)  \right\rangle +\Psi_{t},
\end{array}
\]
where%
\[%
\begin{array}
[c]{ccl}%
\Psi_{t} & = & \left\langle \zeta_{t}\Delta W_{t},k_{t}+f_{y}^{\ast}\left(
t\right)  k_{t}+l_{y}\left(  t\right)  \right\rangle +\left\langle \eta
_{t},\sum_{i=1}^{d}\left[  f_{z_{i}}^{\ast}\left(  t\right)  k_{t}+l_{z_{i}%
}\left(  t\right)  \right]  \Delta W_{t}^{i}\right\rangle \\
&  & +\left\langle k_{t}+f_{y}^{\ast}\left(  t\right)  k_{t}+l_{y}\left(
t\right)  ,\Delta V_{t}\right\rangle +\left\langle \sum_{i=1}^{d}\left[
f_{z_{i}}^{\ast}\left(  t\right)  k_{t}+l_{z_{i}}\left(  t\right)  \right]
\Delta W_{t}^{i},\Delta V_{t}\right\rangle .
\end{array}
\]
It is easy to check that%
\begin{align*}
\mathbb{E}\left\langle \zeta_{t}\Delta W_{t},\sum_{i=1}^{d}f_{z_{i}}^{\ast
}\left(  t\right)  k_{t}\Delta W_{t}^{i}\right\rangle  &  =\mathbb{E}\left[
\sum_{i=1}^{d}\left\langle f_{z_{i}}\left(  t\right)  \zeta_{t}e_{i}%
,k_{t}\right\rangle \right]  ,\\
\mathbb{E}\left\langle \zeta_{t}\Delta W_{t},\sum_{i=1}^{d}l_{z_{i}}\left(
t\right)  \Delta W_{t}^{i}\right\rangle  &  =\mathbb{E}\left[  \sum_{i=1}%
^{d}\left\langle l_{z_{i}}\left(  t\right)  ,\zeta_{t}e_{i}\right\rangle
\right]  .
\end{align*}
Then we have%
\begin{align}
&  \;\;\mathbb{E}\left[  \Delta\left(  \left\langle \xi_{t},p_{t}\right\rangle
+\left\langle \eta_{t},k_{t}\right\rangle \right)  \right]
\label{product-rule-2}\\
&  =\mathbb{E}\left[  \left\langle -b_{x}\left(  t+1\right)  \xi_{t+1}%
,p_{t+1}\right\rangle +\left\langle b_{x}\left(  t\right)  \xi_{t}%
,p_{t}\right\rangle \right. \nonumber\\
&  -\sum_{i=1}^{d}\left\langle \xi_{t+1},\sigma_{ix}^{\ast}\left(  t+1\right)
q_{t+1}e_{i}\right\rangle +\sum_{i=1}^{d}\left\langle \xi_{t},\sigma
_{ix}^{\ast}\left(  t\right)  q_{t}e_{i}\right\rangle \nonumber\\
&  -\left\langle f_{y}\left(  t+1\right)  \eta_{t+1},k_{t+1}\right\rangle
+\left\langle f_{y}\left(  t\right)  \eta_{t},k_{t}\right\rangle \nonumber\\
&  -\sum_{i=1}^{d}\left\langle f_{z_{i}}\left(  t+1\right)  \zeta_{t+1}%
e_{i},k_{t+1}\right\rangle +\sum_{i=1}^{d}\left\langle f_{z_{i}}\left(
t\right)  \zeta_{t}e_{i},k_{t}\right\rangle \nonumber\\
&  +\left\langle l_{x}\left(  t+1\right)  ,\xi_{t+1}\right\rangle
+\left\langle \eta_{t},l_{y}\left(  t\right)  \right\rangle +\sum_{i=1}%
^{d}\left\langle l_{z_{i}}\left(  t\right)  ,\zeta_{t}e_{i}\right\rangle
\nonumber\\
&  +\varepsilon\left\langle \delta_{ts}b_{u}\left(  t\right)  \Delta
v,p_{t}\right\rangle +\delta_{ts}\varepsilon\sum_{i=1}^{d}\left\langle
\sigma_{iu}\left(  t\right)  \Delta v,q_{t}e_{i}\right\rangle \nonumber\\
&  \left.  -\varepsilon\left\langle \delta_{\left(  t+1\right)  s}f_{u}\left(
t+1\right)  \Delta v,k_{t+1}\right\rangle \right]  .\nonumber
\end{align}
Therefore,%
\begin{equation}%
\begin{array}
[c]{cl}
& -\mathbb{E}\left\langle h_{x}\left(  \bar{X}_{T}\right)  ,\xi_{T}%
\right\rangle \\
= & \mathbb{E}\left[  \left\langle \xi_{T},p_{T}\right\rangle +\left\langle
\eta_{T},k_{T}\right\rangle -\left\langle \xi_{0},p_{0}\right\rangle
-\left\langle \eta_{0},k_{0}\right\rangle \right] \\
= & \sum_{t=0}^{T-1}\mathbb{E}\Delta\left(  \left\langle \xi_{t}%
,p_{t}\right\rangle +\left\langle \eta_{t},k_{t}\right\rangle \right) \\
= & \mathbb{E}\left[  \left\langle b_{x}\left(  0\right)  \xi_{0}%
,p_{0}\right\rangle +\sum_{i=1}^{d}\left\langle \xi_{0},\sigma_{ix}^{\ast
}\left(  0\right)  q_{0}e_{i}\right\rangle +\left\langle f_{y}\left(
0\right)  \eta_{0},k_{0}\right\rangle +\sum_{i=1}^{d}\left\langle f_{z_{i}%
}\left(  0\right)  \zeta_{0}e_{i},k_{0}\right\rangle \right] \\
& +\sum_{t=0}^{T-1}\mathbb{E}\left[  \left\langle l_{x}\left(  t\right)
,\xi_{t}\right\rangle +\left\langle l_{y}\left(  t\right)  ,\eta
_{t}\right\rangle +\sum_{i=1}^{d}\left\langle l_{z_{i}}\left(  t\right)
,\zeta_{t}e_{i}\right\rangle \right] \\
& +\sum_{t=0}^{T}\delta_{ts}\varepsilon\mathbb{E}\left[  \left\langle
b_{u}^{\ast}\left(  t\right)  p_{t},\Delta v\right\rangle +\sum_{i=1}%
^{d}\left\langle \sigma_{iu}^{\ast}\left(  t\right)  q_{t}e_{i},\Delta
v\right\rangle -\left\langle f_{u}^{\ast}\left(  t\right)  k_{t},\Delta
v\right\rangle \right]  .
\end{array}
\label{rearrangement}%
\end{equation}
Since $\xi_{0}=0$ and $k_{0}=0$, we deduce%
\begin{equation}%
\begin{array}
[c]{cl}
& \mathbb{E}\sum_{t=0}^{T-1}\left[  \left\langle l_{x}\left(  t\right)
,\xi_{t}\right\rangle +\left\langle l_{y}\left(  t\right)  ,\eta
_{t}\right\rangle +\sum_{i=1}^{d}\left\langle l_{z_{i}}\left(  t\right)
,\zeta_{t}e_{i}\right\rangle \right]  +\mathbb{E}\left\langle h_{x}\left(
\bar{X}_{T}\right)  ,\xi_{T}\right\rangle \\
= & -\varepsilon\mathbb{E}\left[  \left\langle b_{u}^{\ast}\left(  s\right)
p_{s}+\sum_{i=1}^{d}\sigma_{iu}^{\ast}\left(  s\right)  q_{t}e_{i}-f_{u}%
^{\ast}\left(  s\right)  k_{s},\Delta v\right\rangle \right]  .
\end{array}
\label{dual-relation}%
\end{equation}
By $\lim_{\varepsilon\rightarrow0}\frac{1}{\varepsilon}\left[  J\left(
u^{\varepsilon}\left(  \cdot\right)  \right)  -J\left(  \bar{u}\left(
\cdot\right)  \right)  \right]  \geq0$, we obtain%
\[
\mathbb{E}\left[  \left\langle b_{u}^{\ast}\left(  s\right)  p_{s}+\sum
_{i=1}^{d}\sigma_{iu}^{\ast}\left(  s\right)  q_{s}e_{i}-f_{u}^{\ast}\left(
s\right)  k_{s}-l_{u}\left(  s\right)  ,\Delta v\right\rangle \right]  \leq0.
\]
Thus, it is easy to obtain equation (\ref{MP_result_pc}) since $s$ is taking
arbitrarily. This completes the proof.
\end{proof}

\begin{remark}
\label{re-product-rule}In the introduction we point out that we need a
reasonable representation of the product rule. When we calculate
$\Delta\left\langle \xi_{t},p_{t}\right\rangle $ in (\ref{product-rule-1}),
$\Delta\left\langle \xi_{t},p_{t}\right\rangle $ is represented as
$\left\langle \xi_{t+1},\cdot\cdot\cdot\right\rangle +\cdot\cdot\cdot$.
Combining the formulation of the BS$\Delta$E mentioned in the introduction,
this representation will lead to the terms such as $\left\langle \square
_{t},\Diamond_{t}\right\rangle -\left\langle \square_{t+1},\Diamond
_{t+1}\right\rangle $ in (\ref{product-rule-2}). By summing and rearranging
these terms in (\ref{rearrangement}), we obtain the dual relation
(\ref{dual-relation}).
\end{remark}

When $g\equiv0$ and $f\equiv0$, our control system (\ref{c5_eq_w_pc_state_eq}%
)-(\ref{c5_eq_w_pc_cost_eq}) degenerates to the classical discrete control
system which only contains a forward stochastic difference equation as in
\cite{lz15}. For this special case, the adjoint equation becomes%

\begin{equation}
\left\{
\begin{array}
[c]{rcl}%
\Delta p_{t} & = & -b_{x}^{\ast}\left(  t+1\right)  p_{t+1}-\sum_{i=1}%
^{d}\sigma_{ix}^{\ast}\left(  t+1\right)  q_{t+1}e_{i}+l_{x}\left(
t+1\right)  +q_{t}\Delta W_{t}+\Delta Q_{t},\\
p_{T} & = & -h_{x}\left(  \bar{X}_{T}\right)  ,
\end{array}
\right.  \label{zhang}%
\end{equation}
and the Hamiltonian function becomes%
\[
H\left(  \omega,t,u,x,p,q\right)  =b^{\ast}\left(  \omega,t,x,u\right)
p+\sum_{i=1}^{d}\sigma_{i}^{\ast}\left(  \omega,t,x,u\right)  qe_{i}-l\left(
\omega,t,x,u\right)  .
\]
The adjoint equation has the following explicit solution%
\[
\left\{
\begin{array}
[c]{ccl}%
p_{T-1} & = & -\mathbb{E}\left[  h_{x}\left(  \bar{X}_{T}\right)
|\mathcal{F}_{T-1}\right]  ,\\
q_{T-1} & = & -\mathbb{E}\left[  h_{x}\left(  \bar{X}_{T}\right)  \left(
\Delta W_{t}\right)  ^{\ast}|\mathcal{F}_{T-1}\right]  ,\\
p_{t} & = & \mathbb{E}\left[  \left[  I+b_{x}^{\ast}\left(  t+1\right)
\right]  p_{t+1}-l_{x}\left(  t+1\right)  +\sum\limits_{i=1}^{d}\sigma
_{ix}^{\ast}\left(  t+1\right)  q_{t+1}e_{i}|\mathcal{F}_{t}\right]  ,\\
q_{t} & = & \mathbb{E}\left[  \left[  \left[  I+b_{x}^{\ast}\left(
t+1\right)  \right]  p_{t+1}-l_{x}\left(  t+1\right)  +\sum\limits_{i=1}%
^{d}\sigma_{ix}^{\ast}\left(  t+1\right)  q_{t+1}e_{i}\right]  \left(  \Delta
W_{t}\right)  ^{\ast}|\mathcal{F}_{t}\right]  .
\end{array}
\right.
\]
which coincides with the results in \cite{lz15}.

\section{Maximum principle for the fully coupled FBS$\Delta$E system}

In this section we suppose $W$ to be one-dimensional driving process. Let
$\bar{u}=\left\{  \bar{u}_{t}\right\}  _{t=0}^{T}$ be the optimal control for
the control problem (\ref{fbsde_state_eq})-(\ref{fbsde_cost_eq}) and $\left(
\bar{X},\bar{Y},\bar{Z}\right)  $ be the corresponding optimal trajectory.
Note that the existence and uniqueness of $\left(  \bar{X},\bar{Y},\bar
{Z}\right)  $ is guaranteed by the results in \cite{ji-liu}. The perturbed
control $u^{\varepsilon}$ is the same as (\ref{perturb-control}) and we denote
by $\left(  X^{\varepsilon},Y^{\varepsilon},Z^{\varepsilon}\right)
$\thinspace the corresponding trajectory.

Let%
\[
\widehat{X}_{t}=X_{t}^{\varepsilon}-\bar{X}_{t},\text{ }\widehat{Y}_{t}%
=Y_{t}^{\varepsilon}-\bar{Y}_{t},\text{ }\widehat{Z}_{t}=Z_{t}^{\varepsilon
}-\bar{Z}_{t},\text{ }\widehat{N}_{t}=N_{t}^{\varepsilon}-\overline{N}_{t}.
\]

Using the similar notations (\ref{coefficient_notations}) in section 3, we
have%
\begin{equation}
\left\{
\begin{array}
[c]{rcl}%
\Delta\widehat{X}_{t} & = & b^{\varepsilon}\left(  t\right)  -\overline
{b}\left(  t\right)  +\left(  \sigma^{\varepsilon}\left(  t\right)
-\overline{\sigma}\left(  t\right)  \right)  \Delta W_{t},\\
\Delta\widehat{Y}_{t} & = & -f^{\varepsilon}\left(  t+1\right)  +\overline
{f}\left(  t+1\right)  +\widehat{Z}_{t}\Delta W_{t}+\Delta\widehat{N}_{t},\\
\widehat{X}_{0} & = & 0,\\
\widehat{Y}_{T} & = & 0.
\end{array}
\right.  \label{f_c_c_1_difference_1}%
\end{equation}

\begin{lemma}
\label{f_c_c_1_xyz_approxi}Under Assumption \ref{control_assumption} and
Assumption \ref{control_assumption2}, we have%
\begin{equation}
\mathbb{E}\left(  \sum_{t=0}^{T}\left\vert \widehat{X}_{t}\right\vert
^{2}+\sum_{t=0}^{T}\left\vert \widehat{Y}_{t}\right\vert ^{2}+\sum_{t=0}%
^{T-1}\left\vert \widehat{Z}_{t}\right\vert ^{2}\right)  \leq C\varepsilon
^{2}\mathbb{E}\left\vert \Delta v\right\vert ^{2}. \label{x_estimate_1}%
\end{equation}

\end{lemma}

\begin{proof}
By (\ref{f_c_c_1_difference_1}),%
\[%
\begin{array}
[c]{rl}%
0= & \mathbb{E}\left\langle \widehat{X}_{T},\widehat{Y}_{T}\right\rangle
-\mathbb{E}\left\langle \widehat{X}_{0},\widehat{Y}_{0}\right\rangle \\
= & \mathbb{E}\sum_{t=0}^{T-1}\Delta\left\langle \widehat{X}_{t}%
,\widehat{Y}_{t}\right\rangle \\
= & \mathbb{E}\sum_{t=0}^{T}\left[  \left\langle \widehat{X}_{t}%
,-f^{\varepsilon}\left(  t\right)  +\overline{f}\left(  t\right)
\right\rangle +\left\langle \widehat{Y}_{t},b^{\varepsilon}\left(  t\right)
-\overline{b}\left(  t\right)  \right\rangle +\left\langle \widehat{Z}%
_{t},\sigma^{\varepsilon}\left(  t\right)  -\overline{\sigma}\left(  t\right)
\right\rangle \right] \\
= & \mathbb{E}\sum_{t=1}^{T-1}\left\langle A\left(  t,\lambda_{t}%
^{\varepsilon};u_{t}^{\varepsilon}\right)  -A\left(  t,\overline{\lambda}%
_{t};u_{t}^{\varepsilon}\right)  ,\widehat{\lambda}_{t}\right\rangle \\
& +\mathbb{E}\left\langle \widehat{X}_{T},-f^{\varepsilon}\left(  T\right)
+\widetilde{f}^{\varepsilon}\left(  T\right)  \right\rangle +\left\langle
\widehat{Y}_{0},b^{\varepsilon}\left(  0\right)  -\widetilde{b}^{\varepsilon
}\left(  0\right)  \right\rangle +\left\langle \widehat{Z}_{0},\sigma
^{\varepsilon}\left(  0\right)  -\widetilde{\sigma}^{\varepsilon}\left(
0\right)  \right\rangle \\
& +\mathbb{E}\sum_{t=0}^{T}\left[  \left\langle \widehat{X}_{t},-\widetilde{f}%
^{\varepsilon}\left(  t\right)  +\overline{f}\left(  t\right)  \right\rangle
+\left\langle \widehat{Y}_{t},\widetilde{b}^{\varepsilon}\left(  t\right)
-\overline{b}\left(  t\right)  \right\rangle +\left\langle \widehat{Z}%
_{t},\widetilde{\sigma}^{\varepsilon}\left(  t\right)  -\overline{\sigma
}\left(  t\right)  \right\rangle \right] \\
= & \mathbb{E}\sum_{t=1}^{T-1}\left\langle A\left(  t,\lambda_{t}%
^{\varepsilon};u_{t}^{\varepsilon}\right)  -A\left(  t,\overline{\lambda}%
_{t};u_{t}^{\varepsilon}\right)  ,\widehat{\lambda}_{t}\right\rangle \\
& +\mathbb{E}\left\langle \widehat{X}_{T},-f^{\varepsilon}\left(  T\right)
+\widetilde{f}^{\varepsilon}\left(  T\right)  \right\rangle +\left\langle
\widehat{Y}_{0},b^{\varepsilon}\left(  0\right)  -\widetilde{b}^{\varepsilon
}\left(  0\right)  \right\rangle +\left\langle \widehat{Z}_{0},\sigma
^{\varepsilon}\left(  0\right)  -\widetilde{\sigma}^{\varepsilon}\left(
0\right)  \right\rangle \\
& +\mathbb{E}\left[  \left\langle \widehat{X}_{s},-\widetilde{f}^{\varepsilon
}\left(  s\right)  +\overline{f}\left(  s\right)  \right\rangle +\left\langle
\widehat{Y}_{s},\widetilde{b}^{\varepsilon}\left(  s\right)  -\overline
{b}\left(  s\right)  \right\rangle +\left\langle \widehat{Z}_{s}%
,\widetilde{\sigma}^{\varepsilon}\left(  s\right)  -\overline{\sigma}\left(
s\right)  \right\rangle \right]  .
\end{array}
\]
By the monotone condition, we obtain%
\begin{equation}%
\begin{array}
[c]{cl}
& \mathbb{E}\left[  \left\langle \widehat{X}_{s},-\widetilde{f}^{\varepsilon
}\left(  s\right)  +\overline{f}\left(  s\right)  \right\rangle +\left\langle
\widehat{Y}_{s},\widetilde{b}^{\varepsilon}\left(  s\right)  -\overline
{b}\left(  s\right)  \right\rangle +\left\langle \widehat{Z}_{s}%
,\widetilde{\sigma}^{\varepsilon}\left(  s\right)  -\overline{\sigma}\left(
s\right)  \right\rangle \right] \\
\geq & \alpha\mathbb{E}\left[  \sum_{t=0}^{T}\left\vert \widehat{X}%
_{t}\right\vert ^{2}+\sum_{t=0}^{T}\left\vert \widehat{Y}_{t}\right\vert
^{2}+\sum_{t=0}^{T-1}\left\vert \widehat{Z}_{t}\right\vert ^{2}\right]  .
\end{array}
\label{f_c_c_1_difference_2}%
\end{equation}
On the other hand,%
\begin{align*}
&  \;\;\mathbb{E}\left\langle \widehat{X}_{s},-\widetilde{f}^{\varepsilon
}\left(  s\right)  +\overline{f}\left(  s\right)  \right\rangle \\
&  \leq\frac{\alpha}{2}\mathbb{E}\left\vert \widehat{X}_{s}\right\vert
^{2}+\frac{1}{2\alpha}\mathbb{E}\left\vert \overline{f}\left(  s\right)
-\widetilde{f}^{\varepsilon}\left(  s\right)  \right\vert ^{2}\\
&  \leq\frac{\alpha}{2}\mathbb{E}\left\vert \widehat{X}_{s}\right\vert
^{2}+\frac{C}{2\alpha}\varepsilon^{2}\mathbb{E}\left\vert \Delta v\right\vert
^{2}%
\end{align*}
and similarly,%
\begin{equation}%
\begin{array}
[c]{cl}
& \mathbb{E}\left[  \left\langle \widehat{X}_{s},-\widetilde{f}^{\varepsilon
}\left(  s\right)  +\overline{f}\left(  s\right)  \right\rangle +\left\langle
\widehat{Y}_{s},\widetilde{b}^{\varepsilon}\left(  s\right)  -\overline
{b}\left(  s\right)  \right\rangle +\left\langle \widehat{Z}_{s}%
,\widetilde{\sigma}^{\varepsilon}\left(  s\right)  -\overline{\sigma}\left(
s\right)  \right\rangle \right] \\
\leq & \frac{\alpha}{2}\mathbb{E}\left[  \left\vert \widehat{X}_{s}\right\vert
^{2}+\left\vert \widehat{Y}_{s}\right\vert ^{2}+\left\vert \widehat{Z}%
_{s}\right\vert ^{2}\right]  +C\varepsilon^{2}\mathbb{E}\left\vert \Delta
v\right\vert ^{2}.
\end{array}
\label{f_c_c_1_difference_3}%
\end{equation}
Combining (\ref{f_c_c_1_difference_2}) and (\ref{f_c_c_1_difference_3}), we
have%
\[
\mathbb{E}\left[  \sum_{t=0}^{T}\left\vert \widehat{X}_{t}\right\vert
^{2}+\sum_{t=0}^{T}\left\vert \widehat{Y}_{t}\right\vert ^{2}+\sum_{t=0}%
^{T-1}\left\vert \widehat{Z}_{t}\right\vert ^{2}\right]  \leq C\varepsilon
^{2}\mathbb{E}\left\vert \Delta v\right\vert ^{2}.
\]
This completes the proof.
\end{proof}

Next we introduce the following variational equation:%
\begin{equation}
\left\{
\begin{array}
[c]{rcl}%
\Delta\xi_{t} & = & b_{x}\left(  t\right)  \xi_{t}+b_{y}\left(  t\right)
\eta_{t}+b_{z}\left(  t\right)  \zeta_{t}+\delta_{ts}b_{u}\left(  t\right)
\varepsilon\Delta v\\
&  & +\left[  \sigma_{x}\left(  t\right)  \xi_{t}+\sigma_{y}\left(  t\right)
\eta_{t}+\sigma_{z}\left(  t\right)  \zeta_{t}+\delta_{ts}\varepsilon
\sigma_{u}\left(  t\right)  \Delta v\right]  \Delta W_{t},\\
\Delta\eta_{t} & = & -f_{x}\left(  t+1\right)  \xi_{t+1}-f_{y}\left(
t+1\right)  \eta_{t+1}-f_{z}\left(  t+1\right)  \zeta_{t+1}\\
&  & -\delta_{\left(  t+1\right)  s}f_{u}\left(  t+1\right)  \varepsilon\Delta
v+\zeta_{t}\Delta W_{t}+\Delta V_{t},\\
\xi_{0} & = & 0,\\
\eta_{T} & = & 0.
\end{array}
\right.  \label{f_c_c_1_variation}%
\end{equation}

By Assumption \ref{control_assumption} and Assumption
\ref{control_assumption2}, when $t\in\left\{  1,...,T-1\right\}  $,%
\begin{equation}
\left(
\begin{array}
[c]{ccc}%
-f_{x}\left(  t\right)  & -f_{y}\left(  t\right)  & -f_{z}\left(  t\right) \\
b_{x}\left(  t\right)  & b_{y}\left(  t\right)  & b_{z}\left(  t\right) \\
\sigma_{x}\left(  t\right)  & \sigma_{y}\left(  t\right)  & \sigma_{z}\left(
t\right)
\end{array}
\right)  \leq-\alpha I_{3n},\text{ }P-a.s.; \label{f_c_c_1_monotone_1}%
\end{equation}

when $t=0$,%
\begin{equation}
\left(
\begin{array}
[c]{cc}%
b_{y}\left(  0\right)  & b_{z}\left(  0\right) \\
\sigma_{y}\left(  0\right)  & \sigma_{z}\left(  0\right)
\end{array}
\right)  \leq-\alpha I_{2n},\text{ }P-a.s.; \label{f_c_c_1_monotone_2}%
\end{equation}

when $t=T$,%
\begin{equation}
-f_{x}\left(  T\right)  \leq-\alpha I_{n},\text{ }P-a.s..
\label{f_c_c_1_monotone_3}%
\end{equation}

Thus, the coefficients of\thinspace(\ref{f_c_c_1_variation}) satisfy the
monotone condition and there exists a unique solution $\left(  \xi,\eta
,\zeta,V\right)  $ to\thinspace(\ref{f_c_c_1_variation}). Similar to the proof
of Lemma \ref{f_c_c_1_xyz_approxi}, we have%
\begin{equation}
\mathbb{E}\left[  \sum_{t=0}^{T}\left\vert \xi_{t}\right\vert ^{2}+\sum
_{t=0}^{T}\left\vert \eta_{t}\right\vert ^{2}+\sum_{t=0}^{T-1}\left\vert
\zeta_{t}\right\vert ^{2}\right]  \leq C\varepsilon^{2}\mathbb{E}\left\vert
\Delta v\right\vert ^{2}. \label{f_c_c_1_variation_approxi}%
\end{equation}

Define%
\[
\widetilde{\varphi}_{\mu}\left(  t\right)  =\int_{0}^{1}\varphi_{\mu}t,\bar
{X}_{t}+\lambda\left(  X_{t}^{\varepsilon}-\bar{X}_{t}\right)  ,\bar{Y}%
_{t}+\lambda\left(  Y_{t}^{\varepsilon}-\bar{Y}_{t}\right)  ,\bar{Z}%
_{t}+\lambda\left(  Z_{t}^{\varepsilon}-\bar{Z}_{t}\right)  ,\bar{u}%
_{t}+\lambda\left(  u_{t}^{\varepsilon}-\bar{u}_{t}\right)  d\lambda
\]

where $\varphi=b$, $\sigma_{i}$, $g$, $f$, $l$, $h$ and $\mu=x$, $y$, $z_{i}$
and $u$.

\begin{lemma}
\label{f_c_c_1_variation_difference_approxi_0}Under Assumption
\ref{control_assumption} and Assumption \ref{control_assumption2}, we have%
\[
\mathbb{E}\left[  \sum_{t=0}^{T}\left\vert \widehat{X}_{t}-\xi_{t}\right\vert
^{2}+\sum_{t=0}^{T}\left\vert \widehat{Y}_{t}-\eta_{t}\right\vert ^{2}%
+\sum_{t=0}^{T-1}\left\vert \widehat{Z}_{t}-\zeta_{t}\right\vert ^{2}\right]
=o\left(  \varepsilon^{2}\right)  .
\]

\end{lemma}

\begin{proof}
Note that%
\begin{align*}
&  \varphi^{\varepsilon}\left(  t\right)  -\overline{\varphi}\left(  t\right)
\\
&  =\widetilde{\varphi}_{x}\left(  t\right)  \left(  X_{t}^{\varepsilon}%
-\bar{X}_{t}\right)  +\widetilde{\varphi}_{y}\left(  t\right)  \left(
Y_{t}^{\varepsilon}-\bar{Y}_{t}\right)  +\widetilde{\varphi}_{z}\left(
t\right)  \left(  Z_{t}^{\varepsilon}-\bar{Z}_{t}\right)  +\delta
_{ts}\widetilde{\varphi}_{u}\left(  t\right)  \varepsilon\Delta v.
\end{align*}
Set%
\[
\widetilde{X}_{t}=\widehat{X}_{t}-\xi_{t},\text{ }\widetilde{Y}_{t}%
=\widehat{Y}_{t}-\eta_{t},\text{ }\widetilde{Z}_{t}=\widehat{Z}_{t}-\zeta
_{t},\text{ }\widetilde{N}_{t}=\widehat{N}_{t}-V_{t}.
\]
Then,%
\begin{equation}
\left\{
\begin{array}
[c]{rcl}%
\Delta\widetilde{X}_{t} & = & b_{x}\left(  t\right)  \widetilde{X}_{t}%
+b_{y}\left(  t\right)  \widetilde{Y}_{t}+b_{z}\left(  t\right)
\widetilde{Z}_{t}+\Lambda_{1}\left(  t\right) \\
&  & +\left[  \sigma_{x}\left(  t\right)  \widetilde{X}_{t}+\sigma_{y}\left(
t\right)  \widetilde{Y}_{t}+\sigma_{z}\left(  t\right)  \widetilde{Z}%
_{t}+\Lambda_{2}\left(  t\right)  \right]  \Delta W_{t},\\
\Delta\widetilde{Y}_{t} & = & -f_{x}\left(  t+1\right)  \widetilde{X}%
_{t+1}-f_{y}\left(  t+1\right)  \widetilde{Y}_{t+1}-f_{z}\left(  t+1\right)
\widetilde{Z}_{t+1}\\
&  & -\Lambda_{3}\left(  t+1\right)  +\widetilde{Z}_{t}\Delta W_{t}%
+\Delta\widetilde{N}_{t},\\
\widetilde{X}_{0} & = & 0,\\
\widetilde{Y}_{T} & = & 0,
\end{array}
\right.  \label{f_c_c_1_variation_difference_approxi}%
\end{equation}
where%
\begin{align*}
\Lambda_{1}\left(  t\right)   &  =\left(  \widetilde{b}_{x}\left(  t\right)
-b_{x}\left(  t\right)  \right)  \widehat{X}_{t}+\left(  \widetilde{b}%
_{y}\left(  t\right)  -b_{y}\left(  t\right)  \right)  \widehat{Y}_{t}\\
&  +\left(  \widetilde{b}_{z}\left(  t\right)  -b_{z}\left(  t\right)
\right)  \widehat{Z}_{t}+\delta_{ts}\left(  \widetilde{b}_{u}\left(  t\right)
-b_{u}\left(  t\right)  \right)  \varepsilon\Delta v,\\
\Lambda_{2}\left(  t\right)   &  =\left(  \widetilde{\sigma}_{x}\left(
t\right)  -\sigma_{x}\left(  t\right)  \right)  \widehat{X}_{t}+\left(
\widetilde{\sigma}_{y}\left(  t\right)  -\sigma_{y}\left(  t\right)  \right)
\widehat{Y}_{t}\\
&  +\left(  \widetilde{\sigma}_{z}\left(  t\right)  -\sigma_{z}\left(
t\right)  \right)  \widehat{Z}_{t}+\delta_{ts}\left(  \widetilde{\sigma}%
_{u}\left(  t\right)  -\sigma_{u}\left(  t\right)  \right)  \varepsilon\Delta
v,\\
\Lambda_{3}\left(  t\right)   &  =-\left(  \widetilde{f}_{x}\left(  t\right)
-f_{x}\left(  t\right)  \right)  \widehat{X}_{t}-\left(  \widetilde{f}%
_{y}\left(  t\right)  -f_{y}\left(  t\right)  \right)  \widehat{Y}_{t}\\
&  -\left(  \widetilde{f}_{z}\left(  t\right)  -f_{z}\left(  t\right)
\right)  \widehat{Z}_{t}-\delta_{ts}\left(  \widetilde{f}_{u}\left(  t\right)
-f_{u}\left(  t\right)  \right)  \varepsilon\Delta v.
\end{align*}
According to\thinspace(\ref{f_c_c_1_variation_difference_approxi}),%
\[%
\begin{array}
[c]{rl}%
0= & \mathbb{E}\left\langle \widetilde{X}_{T},\widetilde{Y}_{T}\right\rangle
-\mathbb{E}\left\langle \widetilde{X}_{0},\widetilde{Y}_{0}\right\rangle \\
= & \mathbb{E}\sum_{t=0}^{T-1}\Delta\left\langle \widetilde{X}_{t}%
,\widetilde{Y}_{t}\right\rangle \\
= & \mathbb{E}\sum_{t=0}^{T}\left[  \left\langle \widetilde{X}_{t}%
,-f_{\lambda}\left(  t\right)  \widetilde{\lambda}_{t}\right\rangle
+\left\langle \widetilde{Y}_{t},b_{\lambda}\left(  t\right)
\widetilde{\lambda}_{t}\right\rangle +\left\langle \widetilde{Z}_{t}%
,\sigma_{\lambda}\left(  t\right)  \widetilde{\lambda}_{t}\right\rangle
\right] \\
& +\mathbb{E}\sum_{t=0}^{T}\left[  \left\langle \widetilde{X}_{t},-\Lambda
_{3}\left(  t\right)  \right\rangle +\left\langle \widetilde{Y}_{t}%
,\Lambda_{1}\left(  t\right)  \right\rangle +\left\langle \widetilde{Z}%
_{t},\Lambda_{2}\left(  t\right)  \right\rangle \right]
\end{array}
\]
where%
\begin{align*}
\widetilde{\lambda}_{t}  &  =\left(  \widetilde{X}_{t}^{\ast},\widetilde{Y}%
_{t}^{\ast},\widetilde{Z}_{t}^{\ast}\right)  ^{\ast},\\
b_{\lambda}\left(  t\right)   &  =\left(  b_{x}\left(  t\right)  ,b_{y}\left(
t\right)  ,b_{z}\left(  t\right)  \right)  ,\\
\sigma_{\lambda}\left(  t\right)   &  =\left(  \sigma_{x}\left(  t\right)
,\sigma_{y}\left(  t\right)  ,\sigma_{z}\left(  t\right)  \right)  ,\\
f_{\lambda}\left(  t\right)   &  =\left(  f_{x}\left(  t\right)  ,f_{y}\left(
t\right)  ,f_{z}\left(  t\right)  \right)  .
\end{align*}
Combining (\ref{f_c_c_1_monotone_1}), (\ref{f_c_c_1_monotone_2}) and
(\ref{f_c_c_1_monotone_3}), we have%
\begin{equation}%
\begin{array}
[c]{cl}
& \mathbb{E}\sum_{t=0}^{T}\left[  \left\langle \widetilde{X}_{t},-\Lambda
_{3}\left(  t\right)  \right\rangle +\left\langle \widetilde{Y}_{t}%
,\Lambda_{1}\left(  t\right)  \right\rangle +\left\langle \widetilde{Z}%
_{t},\Lambda_{2}\left(  t\right)  \right\rangle \right] \\
\geq & \alpha\mathbb{E}\left[  \sum_{t=0}^{T}\left\vert \widetilde{X}%
_{t}\right\vert ^{2}+\sum_{t=0}^{T}\left\vert \widetilde{Y}_{t}\right\vert
^{2}+\sum_{t=0}^{T-1}\left\vert \widetilde{Z}_{t}\right\vert ^{2}\right]  .
\end{array}
\label{f_c_c_1_variation_difference_1}%
\end{equation}
Note that
\[%
\begin{array}
[c]{rl}
& \mathbb{E}\left\langle \widetilde{X}_{t},-\Lambda_{3}\left(  t\right)
\right\rangle \\
= & \mathbb{E}\left\langle \widetilde{X}_{t},\left(  \widetilde{f}_{x}\left(
t\right)  -f_{x}\left(  t\right)  \right)  \widehat{X}_{t}\right\rangle
+\mathbb{E}\left\langle \widetilde{X}_{t},\left(  \widetilde{f}_{y}\left(
t\right)  -f_{y}\left(  t\right)  \right)  \widehat{Y}_{t}\right\rangle \\
& +\mathbb{E}\left\langle \widetilde{X}_{t},\left(  \widetilde{f}_{z}\left(
t\right)  -f_{z}\left(  t\right)  \right)  \widehat{Z}_{t}\right\rangle
+\mathbb{E}\left\langle \widetilde{X}_{t},\delta_{ts}\left(  \widetilde{f}%
_{u}\left(  t\right)  -f_{u}\left(  t\right)  \right)  \varepsilon\Delta
v\right\rangle \\
\leq & \frac{\alpha}{2}\mathbb{E}\left\vert \widetilde{X}_{t}\right\vert
^{2}+\frac{2}{\alpha}\mathbb{E}\left[  \left\Vert \widetilde{f}_{x}\left(
t\right)  -f_{x}\left(  t\right)  \right\Vert ^{2}\left\vert \widehat{X}%
_{t}\right\vert ^{2}+\left\Vert \widetilde{f}_{y}\left(  t\right)
-f_{y}\left(  t\right)  \right\Vert ^{2}\left\vert \widehat{Y}_{t}\right\vert
^{2}\right] \\
& +\frac{2}{\alpha}\mathbb{E}\left[  \left\Vert \widetilde{f}_{z}\left(
t\right)  -f_{z}\left(  t\right)  \right\Vert ^{2}\left\vert \widehat{Z}%
_{t}\right\vert ^{2}+\delta_{ts}\varepsilon^{2}\left\Vert \widetilde{f}%
_{u}\left(  t\right)  -f_{u}\left(  t\right)  \right\Vert ^{2}\left\vert
\Delta v\right\vert ^{2}\right]  .
\end{array}
\]
When $\varepsilon\rightarrow0$, $\left\Vert \widetilde{f}_{\mu}\left(
t\right)  -f_{\mu}\left(  t\right)  \right\Vert \rightarrow0$\thinspace for
$\mu=x$, $y$, $z$ and $u$. Then, by Lemma \thinspace\ref{f_c_c_1_xyz_approxi},%
\[
\mathbb{E}\left\langle \widetilde{X}_{t},-\Lambda_{3}\left(  t\right)
\right\rangle \leq\frac{\alpha}{2}\mathbb{E}\left\vert \widetilde{X}%
_{t}\right\vert ^{2}+o\left(  \varepsilon^{2}\right)  .
\]
Similar results hold for the other terms in
(\ref{f_c_c_1_variation_difference_1}). Finally, we have%
\[
\mathbb{E}\left[  \sum_{t=0}^{T}\left\vert \widetilde{X}_{t}\right\vert
^{2}+\sum_{t=0}^{T}\left\vert \widetilde{Y}_{t}\right\vert ^{2}+\sum
_{t=0}^{T-1}\left\vert \widetilde{Z}_{t}\right\vert ^{2}\right]  \leq o\left(
\varepsilon^{2}\right)  .
\]
This completes the proof.
\end{proof}

By Lemma \ref{f_c_c_1_variation_difference_approxi_0}, we obtain%
\begin{align*}
&
\begin{array}
[c]{cc}
& J\left(  u^{\varepsilon}\left(  \cdot\right)  \right)  -J\left(  \bar
{u}\left(  \cdot\right)  \right)
\end{array}
\\
&
\begin{array}
[c]{cl}%
= & \mathbb{E}\sum_{t=0}^{T-1}\left[  \left\langle l_{x}\left(  t\right)
,\xi_{t}\right\rangle +\left\langle l_{y}\left(  t\right)  ,\eta
_{t}\right\rangle +\left\langle l_{z}\left(  t\right)  ,\zeta_{t}\right\rangle
+\delta_{ts}\left\langle l_{u}\left(  s\right)  ,\varepsilon\Delta
v\right\rangle \right]  +\mathbb{E}\left\langle h_{x}\left(  \bar{X}%
_{T}\right)  ,\xi_{T}\right\rangle +o\left(  \varepsilon\right)  .
\end{array}
\end{align*}

Introduce the following adjoint equation:%
\begin{equation}
\left\{
\begin{array}
[c]{rcl}%
\Delta p_{t} & = & -b_{x}^{\ast}\left(  t+1\right)  p_{t+1}-\sigma_{x}^{\ast
}\left(  t+1\right)  q_{t+1}\\
&  & +f_{x}^{\ast}\left(  t+1\right)  k_{t+1}+l_{x}\left(  t+1\right)
+q_{t}\Delta W_{t}+\Delta Q_{t},\\
\Delta k_{t} & = & f_{y}^{\ast}\left(  t\right)  k_{t}-b_{y}^{\ast}\left(
t\right)  p_{t}-\sigma_{y}^{\ast}\left(  t\right)  q_{t}+l_{y}\left(  t\right)
\\
&  & +\left[  f_{z}^{\ast}\left(  t\right)  k_{t}-b_{z}^{\ast}\left(
t\right)  p_{t}-\sigma_{z}^{\ast}\left(  t\right)  q_{t}+l_{z}\left(
t\right)  \right]  \Delta W_{t},\\
p_{T} & = & -h_{x}\left(  \bar{X}_{T}\right)  ,\\
k_{0} & = & 0.
\end{array}
\right.  \label{adjoint_eq}%
\end{equation}

Define the Hamiltonian function as follows:%
\[%
\begin{array}
[c]{r}%
H\left(  \omega,t,u,x,y,z,p,q,k\right)  =b^{\ast}\left(  \omega
,t,x,y,z,u\right)  p+\sum_{i=1}^{d}\sigma_{i}^{\ast}\left(  \omega
,t,x,y,z,u\right)  qe_{i}\\
-f^{\ast}\left(  \omega,t,x,y,z,u\right)  k-l\left(  \omega,t,x,y,z,u\right)
.
\end{array}
\]

\begin{theorem}
Suppose that Assumption \ref{control_assumption} and Assumption
\ref{control_assumption2} hold. Let $\bar{u}$ be an optimal control
for\ (\ref{fbsde_state_eq})-(\ref{fbsde_state_eq}), $\left(  \bar{X},\bar
{Y},\bar{Z}\right)  $\ be the corresponding optimal trajectory and $\left(
p,q,k\right)  $ be the solution to the adjoint equation\ (\ref{adjoint_eq}).
Then, for any $t\in\left\{  0,1,...,T\right\}  $ and any\ $v\in U_{t}$, we
have%
\begin{equation}
\left\langle H_{u}\left(  t,\bar{u}_{t},\bar{X}_{t},\bar{Y}_{t},\bar{Z}%
_{t},p_{t},q_{t},k_{t}\right)  ,v-\bar{u}_{t}\right\rangle \leq0,\text{
}P-a.s.. \label{MP_result}%
\end{equation}

\end{theorem}

\begin{proof}
From the expression of $\xi_{t}$, $p_{t}$ for $t\in\left\{
0,1,...,T-1\right\}  $, we have
\[%
\begin{array}
[c]{rl}%
\Delta\left\langle \xi_{t},p_{t}\right\rangle = & \left\langle \xi
_{t+1},\Delta p_{t}\right\rangle +\left\langle \Delta\xi_{t},p_{t}%
\right\rangle \\
= & \left\langle \xi_{t+1},-b_{x}^{\ast}\left(  t+1\right)  p_{t+1}-\sigma
_{x}^{\ast}\left(  t+1\right)  q_{t+1}+f_{x}^{\ast}\left(  t+1\right)
k_{t+1}+l_{x}\left(  t+1\right)  \right\rangle \\
& +\left\langle \left[  \sigma_{x}\left(  t\right)  \xi_{t}+\sigma_{y}\left(
t\right)  \eta_{t}+\sigma_{z}\left(  t\right)  \zeta_{t}+\delta_{ts}%
\varepsilon\sigma_{iu}\left(  t\right)  \Delta v\right]  \Delta W_{t}%
,q_{t}\Delta W_{t}\right\rangle \\
& +\left\langle b_{x}\left(  t\right)  \xi_{t}+b_{y}\left(  t\right)  \eta
_{t}+b_{z}\left(  t\right)  \zeta_{t}+\delta_{ts}b_{u}\left(  t\right)
\varepsilon\Delta v,p_{t}\right\rangle +\Phi_{t},
\end{array}
\]
where%
\[%
\begin{array}
[c]{ccl}%
\Phi_{t} & = & \left\langle \xi_{t}+b_{x}\left(  t\right)  \xi_{t}%
+b_{y}\left(  t\right)  \eta_{t}+b_{z}\left(  t\right)  \zeta_{t}+\delta
_{ts}b_{u}\left(  t\right)  \varepsilon\Delta v,q_{t}\Delta W_{t}\right\rangle
\\
&  & +\left\langle \left[  \sigma_{x}\left(  t\right)  \xi_{t}+\sigma
_{y}\left(  t\right)  \eta_{t}+\sigma_{z}\left(  t\right)  \zeta_{t}%
+\delta_{ts}\varepsilon\sigma_{iu}\left(  t\right)  \Delta v\right]  \Delta
W_{t},p_{t}\right\rangle \\
&  & +\left\langle \xi_{t}+b_{x}\left(  t\right)  \xi_{t}+b_{y}\left(
t\right)  \eta_{t}+b_{z}\left(  t\right)  \zeta_{t}+\delta_{ts}b_{u}\left(
t\right)  \varepsilon\Delta v,\Delta Q_{t}\right\rangle \\
&  & +\left\langle \left[  \sigma_{x}\left(  t\right)  \xi_{t}+\sigma
_{y}\left(  t\right)  \eta_{t}+\sigma_{z}\left(  t\right)  \zeta_{t}%
+\delta_{ts}\varepsilon\sigma_{iu}\left(  t\right)  \Delta v\right]  \Delta
W_{t},\Delta Q_{t}\right\rangle .
\end{array}
\]
Since $W$ and $Q$ are square integrable martingale processes and $Q$ is
strongly orthogonal to $W$, we have $\mathbb{E}\left[  \Phi_{t}\right]  =0$.
Similarly,
\[%
\begin{array}
[c]{rl}%
\Delta\left\langle \eta_{t},k_{t}\right\rangle = & \left\langle \Delta\eta
_{t},k_{t+1}\right\rangle +\left\langle \eta_{t},\Delta k_{t}\right\rangle \\
= & \left\langle -f_{x}\left(  t+1\right)  \xi_{t+1}-f_{y}\left(  t+1\right)
\eta_{t+1}-f_{z}\left(  t+1\right)  \zeta_{t+1}-\delta_{\left(  t+1\right)
s}f_{u}\left(  t+1\right)  \varepsilon\Delta v,k_{t+1}\right\rangle \\
& +\left\langle \zeta_{t}\Delta W_{t},\left[  f_{z}^{\ast}\left(  t\right)
k_{t}-b_{z}^{\ast}\left(  t\right)  p_{t}-\sigma_{z}^{\ast}\left(  t\right)
q_{t}+l_{z}\left(  t\right)  \right]  \Delta W_{t}\right\rangle \\
& +\left\langle \eta_{t},f_{y}^{\ast}\left(  t\right)  k_{t}-b_{y}^{\ast
}\left(  t\right)  p_{t}-\sigma_{y}^{\ast}\left(  t\right)  q_{t}+l_{y}\left(
t\right)  \right\rangle +\Psi_{t},
\end{array}
\]
where%
\[%
\begin{array}
[c]{ccl}%
\Psi_{t} & = & \left\langle \zeta_{t}\Delta W_{t},k_{t}+f_{y}^{\ast}\left(
t\right)  k_{t}-b_{y}^{\ast}\left(  t\right)  p_{t}-\sigma_{y}^{\ast}\left(
t\right)  q_{t}+l_{y}\left(  t\right)  \right\rangle \\
&  & +\left\langle \eta_{t},\left[  f_{z}^{\ast}\left(  t\right)  k_{t}%
-b_{z}^{\ast}\left(  t\right)  p_{t}-\sigma_{z}^{\ast}\left(  t\right)
q_{t}+l_{z}\left(  t\right)  \right]  \Delta W_{t}\right\rangle \\
&  & +\left\langle k_{t}+f_{y}^{\ast}\left(  t\right)  k_{t}-b_{y}^{\ast
}\left(  t\right)  p_{t}-\sigma_{y}^{\ast}\left(  t\right)  q_{t}+l_{y}\left(
t\right)  ,\Delta V_{t}\right\rangle \\
&  & +\left\langle \left[  f_{z}^{\ast}\left(  t\right)  k_{t}-b_{z}^{\ast
}\left(  t\right)  p_{t}-\sigma_{z}^{\ast}\left(  t\right)  q_{t}+l_{z}\left(
t\right)  \right]  \Delta W_{t},\Delta V_{t}\right\rangle .
\end{array}
\]
Furthermore,%
\begin{align*}
&  \;\;\mathbb{E}\left\langle \left[  \sigma_{x}\left(  t\right)  \xi
_{t}+\sigma_{y}\left(  t\right)  \eta_{t}+\sigma_{z}\left(  t\right)
\zeta_{t}+\delta_{ts}\varepsilon\sigma_{iu}\left(  t\right)  \Delta v\right]
\Delta W_{t},q_{t}\Delta W_{t}\right\rangle \\
&  =\mathbb{E}\left\langle \sigma_{x}\left(  t\right)  \xi_{t}+\sigma
_{y}\left(  t\right)  \eta_{t}+\sigma_{z}\left(  t\right)  \zeta_{t}%
+\delta_{ts}\varepsilon\sigma_{iu}\left(  t\right)  \Delta v,q_{t}%
\mathbb{E}\left[  \Delta W_{t}^{2}|\mathcal{F}_{t}\right]  \right\rangle \\
&  =\mathbb{E}\left[  \sum_{i=1}^{d}\left\langle \xi_{t},\sigma_{ix}^{\ast
}\left(  t\right)  q_{t}e_{i}\right\rangle \right]
\end{align*}
and%
\begin{align*}
&  \;\;\mathbb{E}\left\langle \zeta_{t}\Delta W_{t},\left[  f_{z}^{\ast
}\left(  t\right)  k_{t}-b_{z}^{\ast}\left(  t\right)  p_{t}-\sigma_{z}^{\ast
}\left(  t\right)  q_{t}+l_{z}\left(  t\right)  \right]  \Delta W_{t}%
\right\rangle \\
&  =\mathbb{E}\left\langle \zeta_{t},\left[  f_{z}^{\ast}\left(  t\right)
k_{t}-b_{z}^{\ast}\left(  t\right)  p_{t}-\sigma_{z}^{\ast}\left(  t\right)
q_{t}+l_{z}\left(  t\right)  \right]  \mathbb{E}\left[  \Delta W_{t}%
^{2}|\mathcal{F}_{t}\right]  \right\rangle \\
&  =\mathbb{E}\left\langle \sum_{j=1}^{d}\zeta_{t}e_{j}\Delta W_{t}^{j}%
,\sum_{i=1}^{d}f_{z_{i}}^{\ast}\left(  t\right)  k_{t}\Delta W_{t}%
^{i}\right\rangle \\
&  =\mathbb{E}\left[  \sum_{i=1}^{d}\left\langle f_{z_{i}}\left(  t\right)
\zeta_{t}e_{i},k_{t}\right\rangle \right]  .
\end{align*}

Then, we obtain%
\[%
\begin{array}
[c]{rl}
& \mathbb{E}\left[  \Delta\left(  \left\langle \xi_{t},p_{t}\right\rangle
+\left\langle \eta_{t},k_{t}\right\rangle \right)  \right] \\
= & \mathbb{E}\left\langle -b_{x}\left(  t+1\right)  \xi_{t+1},p_{t+1}%
\right\rangle +\left\langle b_{x}\left(  t\right)  \xi_{t},p_{t}\right\rangle
-\left\langle \xi_{t+1},\sigma_{x}^{\ast}\left(  t+1\right)  q_{t+1}%
\right\rangle +\left\langle \xi_{t},\sigma_{x}^{\ast}\left(  t\right)
q_{t}\right\rangle \\
& -\left\langle f_{y}\left(  t+1\right)  \eta_{t+1},k_{t+1}\right\rangle
+\left\langle f_{y}\left(  t\right)  \eta_{t},k_{t}\right\rangle -\left\langle
f_{z}\left(  t+1\right)  \zeta_{t+1},k_{t+1}\right\rangle +\left\langle
f_{z}\left(  t\right)  \zeta_{t},k_{t}\right\rangle \\
& +\left\langle l_{x}\left(  t+1\right)  ,\xi_{t+1}\right\rangle +\left\langle
l_{y}\left(  t\right)  ,\eta_{t}\right\rangle +\left\langle l_{z}\left(
t\right)  ,\zeta_{t}\right\rangle +\varepsilon\left\langle \delta_{ts}%
b_{u}\left(  t\right)  \Delta v,p_{t}\right\rangle +\delta_{ts}\varepsilon
\left\langle \sigma_{u}\left(  t\right)  \Delta v,q_{t}\right\rangle \\
& -\varepsilon\left\langle \delta_{\left(  t+1\right)  s}f_{u}\left(
t+1\right)  \Delta v,k_{t+1}\right\rangle .
\end{array}
\]
Therefore,%
\[%
\begin{array}
[c]{cl}
& -\mathbb{E}\left\langle h_{x}\left(  \bar{X}_{T}\right)  ,\xi_{T}%
\right\rangle \\
= & \mathbb{E}\left[  \left\langle \xi_{T},p_{T}\right\rangle +\left\langle
\eta_{T},k_{T}\right\rangle -\left\langle \xi_{0},p_{0}\right\rangle
-\left\langle \eta_{0},k_{0}\right\rangle \right] \\
= & \sum_{t=0}^{T-1}\mathbb{E}\Delta\left(  \left\langle \xi_{t}%
,p_{t}\right\rangle +\left\langle \eta_{t},k_{t}\right\rangle \right) \\
= & \mathbb{E}\left[  \left\langle b_{x}\left(  0\right)  \xi_{0}%
,p_{0}\right\rangle +\sum_{i=1}^{d}\left\langle \xi_{0},\sigma_{x}^{\ast
}\left(  0\right)  q_{0}\right\rangle +\left\langle f_{y}\left(  0\right)
\eta_{0},k_{0}\right\rangle +\left\langle f_{z}\left(  0\right)  \zeta
_{0},k_{0}\right\rangle \right] \\
& +\sum_{t=0}^{T-1}\mathbb{E}\left[  \left\langle l_{x}\left(  t\right)
,\xi_{t}\right\rangle +\left\langle l_{y}\left(  t\right)  ,\eta
_{t}\right\rangle +\left\langle l_{z}\left(  t\right)  ,\zeta_{t}\right\rangle
\right] \\
& +\sum_{t=0}^{T}\delta_{ts}\varepsilon\mathbb{E}\left[  \left\langle
b_{u}^{\ast}\left(  t\right)  p_{t},\Delta v\right\rangle +\left\langle
\sigma_{u}^{\ast}\left(  t\right)  q_{t},\Delta v\right\rangle -\left\langle
f_{u}^{\ast}\left(  t\right)  k_{t},\Delta v\right\rangle \right]  .
\end{array}
\]
Notice that $\xi_{0}=0$, $k_{0}=0$. So%
\[%
\begin{array}
[c]{cl}
& \mathbb{E}\sum_{t=0}^{T-1}\left[  \left\langle l_{x}\left(  t\right)
,\xi_{t}\right\rangle +\left\langle l_{y}\left(  t\right)  ,\eta
_{t}\right\rangle +\left\langle l_{z}\left(  t\right)  ,\zeta_{t}\right\rangle
\right]  +\mathbb{E}\left\langle h_{x}\left(  \bar{X}_{T}\right)  ,\xi
_{T}\right\rangle \\
= & -\varepsilon\mathbb{E}\left[  \left\langle b_{u}^{\ast}\left(  s\right)
p_{s}+\sigma_{u}^{\ast}\left(  s\right)  q_{s}-f_{u}^{\ast}\left(  s\right)
k_{s},\Delta v\right\rangle \right]  .
\end{array}
\]
Since $\lim_{\varepsilon\rightarrow0}\frac{1}{\varepsilon}\left[  J\left(
u^{\varepsilon}\left(  \cdot\right)  \right)  -J\left(  \bar{u}\left(
\cdot\right)  \right)  \right]  \geq0$, we obtain%
\[
\mathbb{E}\left[  \left\langle b_{u}^{\ast}\left(  s\right)  p_{s}+\sigma
_{u}^{\ast}\left(  s\right)  q_{s}-f_{u}^{\ast}\left(  s\right)  k_{s}%
-l_{u}\left(  s\right)  ,\Delta v\right\rangle \right]  \leq0.
\]
Then, (\ref{MP_result}) holds due to that $s$ is taking arbitrarily. This
completes the proof.
\end{proof}


\begin{thebibliography}{99}                                                                                               %




\bibitem {bz08}Bender, C., \& Zhang, J. (2008). Time discretization and
Markovian iteration for coupled FBSDEs. The Annals of Applied Probability,
18(1), 143-177.

\bibitem {b82}Bensoussan, A. (1982). Lectures on stochastic control. In
Nonlinear filtering and stochastic control (pp. 1-62). Springer, Berlin, Heidelberg.

\bibitem {bcc15}Bielecki, T. R., Cialenco, I., \& Chen, T. (2015). Dynamic
conic finance via backward stochastic difference equations. SIAM Journal on
Financial Mathematics, 6(1), 1068-1122.

\bibitem {b78}Bismut, J. M. (1978). An introductory approach to duality in
optimal stochastic control. SIAM review, 20(1), 62-78.

\bibitem {ce10+}Cohen, S. N., \& Elliott, R. J. (2010). A general theory of
finite state backward stochastic difference equations. Stochastic Processes
and their Applications, 120(4), 442-466.

\bibitem {ce11}Cohen, S. N., \& Elliott, R. J. (2011). Backward stochastic
difference equations and nearly time-consistent nonlinear expectations. SIAM
Journal on Control and Optimization, 49(1), 125-139.

\bibitem {dm06}Delarue, F., \& Menozzi, S. (2006). A forward--backward
stochastic algorithm for quasi-linear PDEs. The Annals of Applied Probability,
16(1), 140-184.

\bibitem {dz99}Dokuchaev, N., \& Zhou, X. Y. (1999). Stochastic controls with
terminal contingent conditions. Journal of Mathematical Analysis and
Applications, 238(1), 143-165.

\bibitem {eh97}El Karoui, N., \& Huang, S. J. (1997). A general result of
existence and uniqueness of backward stochastic differential equations. Pitman
Research Notes in Mathematics Series, 27-38.

\bibitem {hs12}Huang, J., \& Shi, J. (2012). Maximum principle for optimal
control of fully coupled forward-backward stochastic differential delayed
equations*. ESAIM: Control, optimisation and calculus of variations, 18(4), 1073-1096.

\bibitem {hwx09}Huang, J., Wang, G., \& Xiong, J. (2009). A maximum principle
for partial information backward stochastic control problems with
applications. SIAM journal on Control and Optimization, 48(4), 2106-2117.

\bibitem {k72}Kushner, H. J. (1972). Necessary conditions for continuous
parameter stochastic optimization problems. SIAM Journal on Control, 10(3), 550-565.

\bibitem {lz01}Lim, A. E., \& Zhou, X. Y. (2001). Linear-quadratic control of
backward stochastic differential equations. SIAM journal on control and
optimization, 40(2), 450-474.

\bibitem {lz15}Lin, X., \& Zhang, W. (2015). A maximum principle for optimal
control of discrete-time stochastic systems with multiplicative noise. IEEE
Transactions on Automatic Control, 60(4), 1121-1126.

\bibitem {ji-liu}Ji, S. \& Liu, H. (2018). Fully coupled forward-backward
stochastic difference equations, arXiv.

\bibitem {p90}Peng, S. (1990). A general stochastic maximum principle for
optimal control problems. SIAM Journal on control and optimization, 28(4), 966-979.

\bibitem {p93}Peng, S. (1993). Backward stochastic differential equations and
applications to optimal control. Applied Mathematics and Optimization, 27(2), 125-144.

\bibitem {pp90}Pardoux, E., \& Peng, S. (1990). Adapted solution of a backward
stochastic differential equation. Systems \& Control Letters, 14(1), 55-61.

\bibitem {ss99}Schroder, M., \& Skiadas, C. (1999). Optimal consumption and
portfolio selection with stochastic differential utility. Journal of Economic
Theory, 89(1), 68-126.

\bibitem {s10}Stadje, M. (2010). Extending dynamic convex risk measures from
discrete time to continuous time: A convergence approach. Insurance:
Mathematics and Economics, 47(3), 391-404.

\bibitem {ww09}Wang, G., \& Wu, Z. (2009). The maximum principles for
stochastic recursive optimal control problems under partial information. IEEE
Transactions on Automatic control, 54(6), 1230-1242.

\bibitem {w09}Williams, N. (2009). On dynamic principal-agent problems in
continuous time. University of Wisconsin, Madison.

\bibitem {w98}Wu, Z. (1998). Maximum principle for optimal control problem of
fully coupled forward-backward stochastic systems. Systems Science and
Mathematical sciences, 3, 249-259.

\bibitem {w13}Wu, Z. (2013). A general maximum principle for optimal control
of forward--backward stochastic systems. Automatica, 49(5), 1473-1480.

\bibitem {x95}Xu, W. (1995). Stochastic maximum principle for optimal control
problem of forward and backward system. The ANZIAM Journal, 37(2), 172-185.

\bibitem {y10}Yong, J. (2010). Optimality variational principle for controlled
forward-backward stochastic differential equations with mixed initial-terminal
conditions. SIAM Journal on Control and Optimization, 48(6), 4119-4156.

\bibitem {xzx17}Xu, J., Zhang, H., \& Xie, L. (2017). Solvability of general
linear forward and backward stochastic difference equations. Control
Conference (pp.1888-1891). IEEE.

\bibitem {xzx18}Xu, J., Zhang, H., \& Xie, L. (2018). General linear forward
and backward stochastic difference equations with applications. Automatica,
96, 40-50.
\end{thebibliography}
\end{document}